\documentclass[bj]{imsart}

\usepackage{graphicx}
\usepackage{amsmath}
\usepackage{amsthm}
\usepackage{amsfonts}
\usepackage{amssymb}
\usepackage{color}
\usepackage{bbm,subfigure}

\usepackage[all]{xy}

\DeclareMathOperator{\B}{B}

\DeclareMathOperator{\vol}{Vol}

\DeclareMathOperator{\sign}{sign}
\DeclareMathOperator{\conv}{conv}

\def\R{\mathbb{R}}
\def\F{\mathbb{F}}
\def\Z{\mathbb{Z}}
\def\P{\mathbb{P}}

\def\T{\mathbb{T}}

\def\cC{\mathcal{C}}

\def\cG{\mathcal{G}}

\def\cM{\mathcal{M}}

\def\cP{\mathcal{P}}

\def\cU{\mathcal{U}}
\def\cS{\mathcal{S}}

\def\cX{\mathcal{X}}
\def\cY{\mathcal{Y}}

\newcommand{\E}{\mathbb{E}} 

\newcommand{\given}{\;|\;}
\newcommand{\mean}[1] {\E\left\{{#1}\right\}}
\newcommand{\meanx}[1] {\E\{{#1}\}}
\newcommand{\cmean}[2] {\E\left\{#1\given #2\right\}}
\newcommand{\cmeanx}[2] {\E\{#1\given #2\}}
\newcommand{\ind}{\boldsymbol{\mathbbm{1}}} 

\newcommand{\var}[1]{\mathrm{Var}\param{{#1}}}
\newcommand{\varx}[1]{\mathrm{Var}({#1})}

\newcommand{\iprod}[1]{\left\langle{#1}\right\rangle}
\newcommand{\set}[1]{\left\{#1\right\}}

\newcommand{\norm}[1]{\left\|#1\right\|}
\newcommand{\param}[1]{\left(#1\right)}
\newcommand{\abs}[1] {\left| {#1}\right|}
\newcommand{\floor}[1] {\left\lfloor{#1}\right\rfloor}

\newcommand{\prob}[1]{\mathbb{P}\left(#1\right)}

\newcommand{\cprob}[2]{\mathbb{P}\left(#1\given #2\right)} 

\newcommand{\eps}{\epsilon}

\newcommand{\by}{\mathbf{y}}
\newcommand{\bx}{\mathbf{x}}

\def\bth{\boldsymbol\theta}
\providecommand{\setthms}[1]{#1}
\setthms{
\newtheorem{lem}{Lemma}[section]
\newtheorem{thm}[lem]{Theorem}
\newtheorem{prop}[lem]{Proposition}
\newtheorem{cor}[lem]{Corollary}

\theoremstyle{definition}
\newtheorem{defn}[lem]{Definition}

}
\newcommand{\cech}{\v{C}ech }
\newcommand{\erdos}{Erd\H{o}s }
\newcommand{\renyi}{R\'enyi }

\newcommand{\iid}{\mathrm{i.i.d.}}

\newcommand{\ninf}{n\to\infty}
\newcommand{\pois}[1]{\mathrm{Poisson}\param{{#1}}}

\newcommand{\limninf}{\lim_{\ninf}}

\numberwithin{equation}{section}

\begin{document}

\begin{frontmatter}

\title{On the Vanishing of Homology in Random \v{C}ech Complexes}
\runtitle{Vanishing of Homology in Random \v{C}ech Complexes}

\begin{aug}
  \author{\fnms{Omer}  \snm{Bobrowski}\thanksref{m1}\ead[label=e1]{omer@math.duke.edu}} and
  \author{\fnms{Shmuel} \snm{Weinberger}\ead[label=e2]{shmuel@math.uchicago.edu}}\thanksref{m2}
  \runauthor{O. Bobrowski \& S. Weinberger}

	
  \address[m1]{Department of Mathematics, Duke University\\
          \printead{e1}}

  \address[m2]{Departments of Mathematics, University of Chicago\\
          \printead{e2}} 

\end{aug}


\begin{abstract}
We compute the homology of random \cech complexes over a homogeneous Poisson process on
the $d$-dimensional torus, 
and show that there are, coarsely, two phase transitions.  The first transition is analogous to the \erdos-\renyi phase transition, where the \cech complex becomes connected. The second transition is where all the other homology groups are computed correctly (almost simultaneously).  Our calculations also suggest a finer measurement of scales, where there is a further refinement to this picture and separation between different homology groups.
\end{abstract}


\end{frontmatter}

\maketitle


\section{Introduction}


This paper is a continuation of three different stories. The first story begins with the work of \erdos and \renyi \cite{erdos_random_1959} on the random graphs $\cG(n,p)$ (with $n$ vertices, and where every edge is included independently with probability $p$), and their connectivity properties (see also \cite{bollobas_random_1998}). It continues with Penrose's work \cite{penrose_random_2003} extending this to the geometric graphs $\cG(\cP,r)$ -  graphs whose vertices are random point processes ($\cP$) in a $d$-dimensional space, and whose edges are determined by the proximity between the points (distance less than $r$). In both models, a phase transition occurs when the number of vertices $n$ goes to infinity, and the average degree is of order $\log n$ ($p=\frac{\log n}{n}$, or $r \propto (\frac{\log n}{n})^{1/d}$). In this sharp phase transition, if the degree is lower then a known threshold the graph has many connected components, but once the threshold is passed the graphs become connected.

The second story begins much later in the work of Linial and Meshulam \cite{linial_homological_2006}, and is about high dimensional extensions of the \erdos-\renyi theory, where instead of a graph one studies a random simplicial complex. See \cite{aronshtam_when_2015,aronshtam_collapsibility_2013,babson_fundamental_2011,cohen_topology_2012,costa_asphericity_2015,kahle_sharp_2014,kahle_inside_2014,kozlov_threshold_2010,linial_homological_2006,linial_phase_2014} for a variety of important results. Random simplicial complexes have various models, and much is still unknown about those, although remarkable phenomena have been found regarding their homology (a higher dimensional generalization of connectedness, see Section \ref{sec:homology}) and their fundamental groups.
A direct motivation for the work presented in this paper can be found in Kahle's work \cite{kahle_sharp_2014} describing the vanishing of homology of random flag complexes. Random flag complexes are generated from  \erdos-\renyi graphs by adding a $k$-simplex for any $(k+1)$-cliques in the graph. The main result in \cite{kahle_sharp_2014} states that the $k$-th homology group vanishes (becomes trivial) when $p = \param{\frac{(\frac{k}{2}+1)\log n}{n}}^{1/(k+1)}$, and the result of \erdos and \renyi can be viewed as the special case $k=0$. Our goal in this paper is to explore an equivalent extension in the random geometric graph context.

The third story is related to the theory of coverage processes (cf.~\cite{hall_introduction_1988}).  In particular, Flatto and Newman \cite{flatto_random_1977} have shown that on a $d$-dimensional unit volume Riemannian manifold $\cM$, generating $n$ uniformly distributed points and covering them with balls of radius $r$ will cover $\cM$ with high probability provided that  $r\propto (\frac{\log n}{n})^{1/d}$ (resembling the connectivity threshold discussed above). On first sight, this phase transition might seem unrelated to the problem at hand, but as we shall see soon, it provides significant information about the vanishing of homology in geometric complexes.

There are various ways to define geometric simplicial complexes. The one studied in this paper is the \emph{random \cech complex}. In this model we generate points randomly in a metric space (e.g.~a compact Riemannian manifold $\cM$), and form a simplicial complex by fixing a radius $r$ and asserting that $k$ points span a $(k-1)$-simplex if the $r$-balls around these points have a nonempty intersection. 
The theory of random geometric complexes (such as the \cech complex) has been growing rapidly in the past decade (cf. \cite{adler_persistent_2010,adler_crackle:_2014,bobrowski_distance_2014,bobrowski_topology_2014,kahle_random_2011,kahle_limit_2013,owada_limit_2015,yogeshwaran_topology_2012,yogeshwaran_random_2014} and the comprehensive survey in \cite{bobrowski_topology_2014-1}). In particular, the limiting behavior of the homology of these complexes has been studied, when $n\to\infty$ and $r=r(n)\to 0$.

In this paper we study the following model. Let $\T^d = \R^d/ \Z^d$ be the $d$-dimensional flat cubical torus (see Section \ref{sec:flat_torus}),
and let $\cP_n$ be a homogeneous Poisson process in $\T^d$ with intensity $n$ (i.e.~$\mean{\abs{\cP_n}} = n$).  We are interested in phase transitions related to the homology of the \cech complex $\cC(\cP_n,r)$ as $n\to\infty$ and $r=r(n) \to 0$. The limiting behavior of the complex is controlled by the term $\Lambda := \omega_d nr^d$, where $\omega_d$ is the volume of a $d$-dimensional Euclidean unit ball. This quantity can be viewed as measuring the expected degree of the underlying graph.

For this random geometric complex, we shall see that essentially there are two sharp phase transitions. The first transition is the \erdos\!\!--\renyi\!\!--Penrose connectedness type, which occurs when $\Lambda = 2^{-d}\log n$. The second transition is related to the rest of the homology groups, and occurs when $\Lambda = \log n$.
Below the second threshold, we see a lot of extra homology (cycles, or holes in various dimensions $k<d$), while above the threshold the homology of the random complex becomes identical to that of the torus (a phenomenon we refer to as the ``vanishing" of homology).

As mentioned earlier, a phenomenon related to the vanishing phase transition has also been observed in the literature on coverage processes (see \cite{hall_introduction_1988}).  Flatto and Newman \cite{flatto_random_1977} have shown that $\Lambda = \log n + (d-1)\log\log n$ is the threshold for the balls or radius $r$ around $\cP_n$ to cover a $d$-dimensional manifold
(see also Appendix \ref{sec:appendix_coverage}). Since the union of balls and the \cech complex share the same homology (see Section \ref{sec:cech}), and since coverage implies the vanishing of homology, the result in \cite{flatto_random_1977} provides an upper bound to the vanishing threshold we are interested at.

Our main results (see Section \ref{sec:main_results}) assert that when $\Lambda = (1+\eps)\log n $ the homology of $\cC(\cP_n,r)$ is identical to that of $\T^d$, while prior to this point, i.e.~when $\Lambda = (1-\eps)\log n$, the homology groups of $\cC(\cP_n,r)$ are very large (the number of cycles grows to infinity).
To prove this statement we show that at the moment before coverage many nontrivial cycles are still being formed. We do this using Morse theory for distance functions (see e.g.~\cite{cheeger_critical_1991,gershkovich_morse_1997}, and Section \ref{sec:morse_dist}).  In Morse theory (see Section \ref{sec:morse}), one counts critical points of functions according to the index of the Hessian (i.e.~the matrix of second order derivatives) at the critical points, with an aim towards computing homology.  The main idea is that a critical point of index $k$ may either give rise to a $k$-dimensional cycle or alternatively signal the death of a $(k-1)$-dimensional cycle. In general, however, there is no rule to determine which of these takes place. As a result, Morse theory ordinarily does not provide an accurate analysis of  homology, but rather a set of inequalities.  In this paper on the other hand, we construct a special type of critical points that are guaranteed to generate cycles in the dimension of their index. We refer to these cycles as ``$\Theta$-cycles" (see Section \ref{sec:theta_cycles}), and use them to show the existence of non-trivial homology. 

The study of critical points leads to some additional information about the finer scale structure.  We see that analogously to the results of \cite{flatto_random_1977},  at an additive $\log\log n$ scale, there seems to be a separation between the vanishing of the different dimensions of homology.  In particular, our results (see Theorem \ref{thm:homology_bounds}) suggest that the exact vanishing threshold of the homology in dimension $k$ is in the range
\[
	[\log n +(k-2)\log\log n, \log n + k\log\log n].
\]
For example, at this fine scale the $1$-dimensional cycles vanish before the $3$-dimensional ones (with high probability).
It is very interesting but beyond the scope of our current methods to try to understand exactly where in these finer scale ranges the various homology groups converge to those of the torus, and the nature of the vanishing results that might occur at that scale.
 
While random geometric complexes provide grounds for rich and deep theoretical probability research, we note that they also find applications in data analysis and network modeling. The rapidly developing field known as \emph{Topological Data Analysis} (TDA) focuses on using topological signatures of data in machine learning and statistics (for some introduction see \cite{carlsson_topology_2009,ghrist_barcodes:_2008,zomorodian_topological_2007}). Geometric complexes play a key role in the conversion of abstract topological questions into a simplified set of algebraic and combinatoric operations that can be coded in software.
Analyzing the behavior of random geometric complexes is therefore imperative in order to provide TDA with rigorous statistical statements  (see e.g.~\cite{chazal_sampling_2009,niyogi_finding_2008,niyogi_topological_2011}). The results in this paper provide information that is asymptotically an improvement of the few estimates on the sample complexity in topological inference problems.  It is also related to \cite{balakrishnan_minimax_2012,balakrishnan_tight_2013}  on the rate of convergence for homology estimation in riemannian manifolds.  

The structure of this paper is as follows. Sections 2-4 provide the necessary background and definitions for this work. The main results are presented in Section 5, describing the phase transition for the vanishing  of homology. Sections 6 and 7 present the main ideas and lemmas used to prove the upper and lower bounds, respectively.
Section 8 presents the detailed proofs for the statements in Sections 5-7. 

Finally, while the results of this paper are proved explicitly only for the case of the cubical torus, they trivially apply to any flat torus (of unit volume, i.e. $\R^d/L$ for any lattice $L$ in $\R^d$). In section 9 we give a heuristic (ineffective) argument why they also should apply to arbitrary closed Riemannian manifolds (normalized to have unit volume).  We hope to return to this question with effective estimates in a later paper.


\section{Topological background}

In this section we wish to provide a brief introduction to homology, \cech complexes and Morse theory, which we will use later in this paper.


\subsection{Homology}\label{sec:homology}


We wish to introduce the concept of homology here in an intuitive rather than a rigorous way. A comprehensive introduction to the topic can be found in \cite{hatcher_algebraic_2002,munkres_elements_1984}.
Let $X$ be a topological space, the \textit{homology} of $X$ is a sequence of abelian groups denoted  $\set{H_i(X)}_{i=0}^\infty$. Homology is a \emph{topological invariant}, namely if $f:X\to Y$ is a homeomorphism, then it induces an isomorphism $f_*:H_*(X)\to H_*(Y)$ between the homology groups.
 
In the case where homology is computed using coefficients in a field  $\mathbb{F}$, then $H_i(X)$ is simply a vector space\footnote{We introduce homology with field coefficients for simplicity. Our results, however, apply to homology with arbitrary coefficients, not just fields.}. The  basis elements of zeroth homology $H_0(X)$  correspond to the connected components of $X$. For example, if $X$ has three connected components, then $H_0(X) \cong \mathbb{F}^3$ (where $\cong$ denotes  isomorphism), and each of the basis element corresponds to a different connected component of $X$. For $k\ge 1$, the basis elements of the $k$-th homology $H_k(X)$ correspond to  $k$-dimensional ``holes" or (nontrivial) ``cycles" in $X$. An intuitive way to think about a $k$-dimensional cycle is as the boundary of a $(k+1)$-dimensional body.
For example, if $X$ a circle then $H_1(X) \cong \mathbb{F}$, if $X$ is a $2$-dimensional sphere then $H_2(X) \cong \mathbb{F}$, and in general if $X$ is a $d$-dimensional sphere, then
\[
H_k(X) \cong \begin{cases} \mathbb{F} & k=0,d \\
\set{0} & \mbox{otherwise}.
\end{cases}
\]
Another example which will be relevant in this paper is the torus. The $2$-dimensional torus $\mathbb{T}^2$ (see Figure \ref{fig:flat_torus}) has a single connected component and a single $2$-dimensional hole (the void inside the surface). This implies that $H_0(\T^2) \cong \F$, and $H_2(\T^2) \cong \F$. As for $1$-cycles (or closed loops) the torus has two different loops, and so $H_1({\mathbb{T}^2}) \cong\F^2$. In this paper we will consider the general $d$-dimensional torus denoted by $\T^d$. In that case it turns out that $H_k(\T^d) \cong \F^{\binom{d}{k}}$.

The last term we want to introduce here is {homotopy equivalence}. 
Let $X,Y$ be two topological spaces. A $\emph{homotopy}$ is a continuous function $F:X\times[0,1]\to Y$, which can be viewed as a indexed sequence of functions $F(\cdot,t):X\to Y$.
Two functions $f_0,f_1:X\to Y$ are called \emph{homotopic} (denoted $f_0\simeq f_1$) if there exists a homotopy $F$ such that $F(\cdot, 0) = f_0$ and $F(\cdot,1) = f_1$. Finally, $f:X\to Y$ is called a \emph{homotopy equivalence} if there exists $g:Y\to X$ such that $g\circ f = \mathrm{id}_X$ and $f\circ g \simeq \mathrm{id}_Y$ ($\mathrm{id}_X$ refers to the identity function on $X$). If there exists a homotopy equivalence $f:X\to Y$ then we say that $X$ and $Y$ are \emph{homotopy equivalent}, and denote it by $X\simeq Y$.
Loosely speaking, $X\simeq Y$ means that we can continuously transform one of the spaces into the other, generalizing the notion of homeomorphic spaces (for example, a ball and a point are homotopy equivalent but not homeomorphic). In particular, if $X\simeq Y$ then it can be shown that they have the same homology, i.e.~$H_k(X)\cong H_k(Y)$ for all $k$. A space that is homotopy equivalent to a single point is called \emph{contractible}.


\subsection{\cech complexes}\label{sec:cech}


For a given set $S$, an \emph{abstract simplicial complex} $\Sigma$ on $S$ is a collection of finite subsets $A\in 2^S$ such that if $A\in \Sigma$ and $B\subset A$ then also $B\in \Sigma$.
We refer to  sets $A\in\Sigma$ with $\abs{A} = k+1$ as the $k$-simplexes or $k$-faces in $\Sigma$.
In this paper we study an abstract simplicial complex known as the \cech complex, defined in the following way.


\begin{defn}[\cech complex]\label{def:cech_complex}
Let $\cP = \set{x_1,x_2,\ldots,x_n}$ be a collection of points in a metric space $(X,\rho)$, and let $r>0$ and let $B_r(x)$ be the ball of radius $r$ around $x$. The \cech complex $\cC(\cP, r)$ is constructed as follows:
\begin{enumerate}
\item The $0$-simplexes (vertices) are the points in $\cP$.
\item A $k$-simplex $[x_{i_0},\ldots,x_{i_k}]$ is in $\cC(\cP,r)$ if $\bigcap_{j=0}^{k} {B_{r}(x_{i_j})} \ne \emptyset$.

\end{enumerate}
\end{defn}



\begin{figure}[h]
\centering
\includegraphics[scale=0.4]{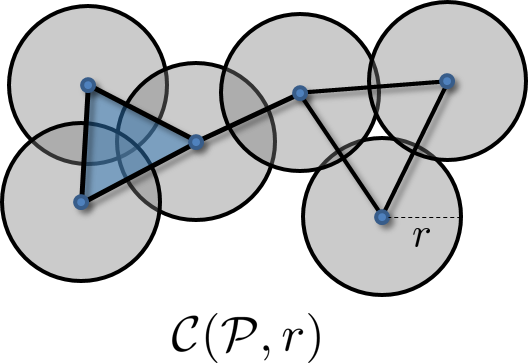}
\caption{\label{fig:cech}  A \cech complex generated by a set of points in $\R^2$. The complex has 6 vertices (0-simplexes), 7 edges (1-simplexes) and one triangle (a 2-simplex). }
\end{figure}

Associated with the \cech complex $\cC(\cP,r)$ is the union of balls used to generate it, which we define as
\begin{equation}\label{eq:union_balls}
	\cU(\cP, r) = \bigcup_{p\in \cP}B_r(p).
\end{equation}
The spaces $\cC(\cP,r)$ and $\cU(\cP,r)$ are of a completely different nature. Nevertheless, the following lemma claims that topologically they are very similar.
This lemma a special case of a more general topological statement originated in \cite{borsuk_imbedding_1948} and commonly referred to as the `Nerve Lemma'.


\begin{lem}\label{lem:nerve}
Let $\cC(\cP,r)$ and $\cU(\cP,r)$ as defined above. If for every $x_{i_1},\ldots,x_{i_k}$ the intersection $B_r(x_{i_1})\cap \cdots\cap B_r(x_{i_k})$ is either empty or contractible (homotopy equivalent to a point), then $\cC(\cP,r)\simeq \cU(\cP,r)$, and in particular,
\[
	H_k(\cC(\cP,r)) \cong H_k(\cU(\cP,r)),\quad \forall k\ge 0.
\]
\end{lem}


Consequently, we will sometimes be using $\cU(\cP,r)$ to make statements about $\cC(\cP,r)$. This will be very useful especially when coverage arguments are available. Note that in Figure \ref{fig:cech} indeed both $\cC(\cP,r)$ and $\cU(\cP,r)$ have a single connected component and a single hole.


\subsection{Morse theory for the distance function}\label{sec:morse_dist}


The distance function defined below and its critical points will play a key role in our analysis of the vanishing threshold for homology. In this section we wish to define the distance function and its critical points, and briefly introduce Morse theory.

Let $\cP$ be a finite set of points in a metric space $(X,\rho)$. We define the distance function from $\cP$ as follows -
\begin{equation}\label{eq:dist_fn}
	\rho_{_\cP}(x) = \min_{p\in \cP} \rho(x,p),\quad \forall x\in X.
\end{equation}
Our  interest in this function stems from the following straightforward observation about the sublevel sets of the distance function:
\[
	\rho_{_\cP}^{-1}((-\infty,r]) = 	\rho_{_\cP}^{-1}([0,r]) = \cU(\cP,r).
\]
In other words, the sublevel sets of the distance functions are exactly the union of balls associated with the \cech complex.
The idea of Morse theory is to link the study of critical points of functions with the changes to the homology of their sublevel sets.
In analyzing the vanishing of homology for the \cech complex, this link will show up to be highly useful.

\subsubsection{Critical points of the distance function}\label{sec:crit_pts}

The classical definition of critical points in calculus is as follows. Let $f:\R^d\to\R$ be a $C^2$ function. A point $c\in \R$ is called a \textit{critical point} of $f$ if $\nabla f (c) =0$, and the real number $f(c)$ is called a \textit{critical value} of $f$. A critical point $c$ is called \textit{non-degenerate} if the Hessian matrix $H_f(c)$ is non-singular. In that case, the \textit{Morse index} of $f$ at $c$, denoted by $\mu(c)$ is the number of negative eigenvalues of $H_f(c)$. 

Note that the distance function $\rho_{_\cP}$ defined in \eqref{eq:dist_fn} is not everywhere differentiable, therefore the definition above does not apply. To overcome this problem we will use \emph{Morse theory for min-type functions} which was developed in \cite{gershkovich_morse_1997}.  Min-type functions are of the form: $f(x) = \min_{1\le i \le m} \alpha_i(x)$ where $\alpha_1,\ldots,\alpha_m$ are smooth functions. The work in \cite{gershkovich_morse_1997} provides  the definitions of regular and critical points for this type of functions, and shows that large parts of Morse theory (discussed below) can be applied to these functions as well.

In this section we wish to introduce the definitions of critical points for $\rho_{_\cP}$ for the special case where $\cP\subset\R^d$ (with the Euclidean metric $\rho(x,y) = \norm{x-y}$). Later on, we will see how to extend these ideas to the $d$-dimensional torus.
Note that the distance function $\rho_{_\cP}$ does not fall directly into this category of min-type functions, since the function $\rho(p,\cdot)$ is non-differentiable at $p$. However, the squared-distance function $\rho_{_\cP}^2$ is a true min-type function for which the results in \cite{gershkovich_morse_1997} apply. The sublevel sets of $\rho_{_\cP}^2$ are still union of balls, i.e.~$(\rho_{_\cP}^2)^{-1}((-\infty,r]) = \cU(\cP, \sqrt{r})$, and therefore any conclusion we can make using Morse theory for $\rho_{_\cP}^2$ could be naturally translated to similar statements about the sublevel sets of $\rho_{_\cP}$. Hence, from here on, we will talk about critical points and Morse theory for $\rho_{_\cP}$, while formally we mean $\rho_{_\cP}^2$.

The following definitions for the critical points of $\rho_{_\cP}$  appeared already in \cite{bobrowski_distance_2014}, and we include them here for completeness. They are merely an adaptation of the general definitions appeared in \cite{gershkovich_morse_1997} (Section 3.1) to the distance function case.
Start with the local (and global) minima of $\rho_{_\cP}$, which are all the points in 
$\cP$ where $\rho_{_\cP} = 0$, and call these
\emph{critical points with index $0$}. For higher indices,
we have  the following.

\begin{defn}\label{def:crit_pts}
A point $c\in\R^d$ is \emph{a critical point of  $\rho_{_\cP}$ with index $1 \le k \le d$} if there exists a subset $\cY\subset\cP$ of $k+1$ points such that:
\begin{enumerate}
\item $\forall y\in \cY:  \rho_{_\cP}(c) = \norm{c-y} $, and, $\forall p\in \cP \backslash \cY$ we have $\norm{c-p} > \rho_{_\cP}(p)$.
\item The points in $\cY$ are in general position (i.e.\! the $k+1$ points
of $\cY$ do not lie in a $(k-1)$-dimensional affine space).
\item The point $c$ lies inside the  open $k$-simplex spanned by $\cY$.
\end{enumerate}
\end{defn}
The first condition ensures that $\rho_{_\cP} \equiv \rho_{_\cY}$ in a neighborhood of $c$.
The second condition implies that the set  $\cY$ lies on a unique $(k-1)$- dimensional sphere.
We note that this generality requirement is satisfied almost surely when the points are generated by a Poisson process.
 We shall use the following notation:
\begin{align}
S(\cY) &= \textrm{The unique $(k-1)$-dimensional sphere containing $\cY$},\label{eq:def_S}\\
C(\cY) &= \textrm{The center of $S(\cY)$ in $\R^d$}, \label{eq:def_C}\\
R(\cY) &= \textrm{The radius of $S(\cY)$} , \label{eq:def_R}\\
B(\cY) &= \textrm{The open ball in $\R^d$ with radius $R(\cY)$ centered at $C(\cY)$}, \label{eq:def_B} \textrm{and} \\
\Delta(\cY) &= \textrm{The open $k$-simplex spanned by the $k+1$ points in $\cY$}.
\label{eq:def_Delta}
\end{align}
Note that $S(\cY)$ is a $(k-1)$-dimensional sphere, whereas $B(\cY)$ is a $d$-dimensional ball. Obviously, $S(\cY) \subset \partial B(\cY)$, but unless $k=d$, $S(\cY)$ is
{\it not} the boundary of $B(\cY)$.
Since the critical point $c$ in Definition \ref{def:crit_pts} is equidistant from all the points in $\cY$, we have that $c=C(\cY)$. Thus, we say that $c$ is the unique index $k$ critical point \textit{generated} by the $k+1$ points in $\cY$.
The last statement can be rephrased as follows:


\begin{lem} \label{lem:gen_crit_point}
A subset $\cY\subset \cP$ of $k+1$ points in general position generates an index $k$ critical point if and only if the following two conditions hold:
\[
	(1)\ \ C(\cY) \in \Delta(\cY),\quad\textrm{ and }\quad (2)\ \ \cP \cap {B(\cY)}= \emptyset
\]
The critical point generated is $C(\cY)$, and the critical value is $R(\cY) = \rho_{_\cP}(C(\cY))$.
\end{lem}


Figure \ref{fig:crit_pts} depicts the generation of an index $2$ critical point in $\R^2$ by  subsets of $3$ points.


\begin{figure}[h]
\begin{center}
\includegraphics[scale=0.3]{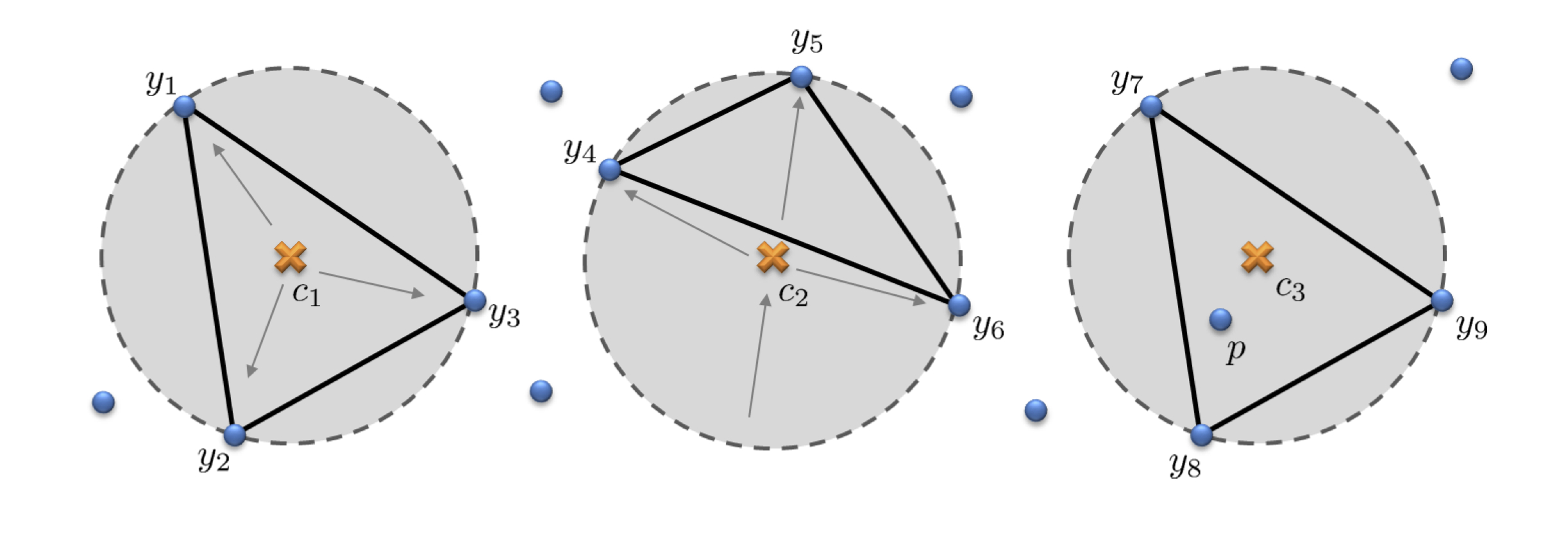}
\caption{\label{fig:crit_pts} Generating a critical point of index $2$ in $\R^2$, i.e.~a maximum point.
The small blue disks are the points of  $\cP$. We examine three subsets of $\cP$: $\cY_1 = \set{y_{_1},y_{_2},y_{_3}}$, $\cY_2 = \set{y_{_4},y_{_5},y_{_6}}$, and $\cY_3 = \set{y_{_7},y_{_8},y_{_9}}$. $S(\cY_i)$ are the dashed circles, whose centers are $C(\cY_i) = c_i$. The shaded balls are $B(\cY_i)$, and the interior of the triangles are $\Delta(\cY_i)$. The arrows represent the flow direction. For $\cY_1$ both conditions in Lemma \ref{lem:gen_crit_point} hold and therefore $c_1$ is a critical point. However, for $\cY_2$ condition (1) does not hold, and for $\cY_3$ condition (2) fails, therefore $c_2$ and $c_3$ are not critical points.}
\end{center}
\end{figure}


\subsubsection{Morse Theory}\label{sec:morse}


The study of homology is strongly connected to the study of critical points of real valued functions. The link between them is called Morse theory, and we shall describe it here briefly. For a more detailed introduction, we refer the reader to \cite{milnor_morse_1963}.

The main idea of Morse theory is as follows. Suppose that $\cM$ is a closed manifold (a compact manifold without boundary), and let $f:\cM\to \R$ be a Morse function.
Denote
\[
\cM_{v} := f^{-1}((-\infty,v]) = \set{x\in \cM : f(x) \le v}\subset \cM
\] (sublevel sets of $f$).
If there are no critical values in $(a,b]$, then $\cM_a$ and $\cM_b$ are homotopy equivalent and in particular $H_k(\cM_a)\cong H_k(\cM_b)$ for all $k$.
Next, suppose that $c$ is a critical point of $f$ with Morse index $k$, and let $v=f(c)$ be the critical value at $c$. Then at $\cM_{v}$  homology changes in the following way. For a small enough $\eps$ we have that the homology of $\cM_{v+\eps}$ is obtained from the homology of $\cM_{v-\eps}$ by either adding  a generator to $H_k$ (increasing its dimension by one) or removing a generator of $H_{k-1}$ (decreasing its dimension by one). In other words, as we pass a critical value, either a new $k$-dimensional cycle is formed, or an existing $(k-1)$-dimensional cycle is terminated.

While classical Morse theory deals with smooth (or $C^2$) Morse functions on compact manifolds \cite{milnor_morse_1963} it has many generalizations, and the extension to ``min-type'' functions presented in \cite{gershkovich_morse_1997} enables one to apply similar concepts to the distance function $\rho_{_\cP}$ as well. In particular,  the critical points  defined in Section \ref{sec:crit_pts} have similar effect on the homology as in classical Morse theory.


\section{The $d$-dimensional flat torus}\label{sec:flat_torus}


The results in this paper should apply  to homogeneous Poisson processes generated on any compact Riemannian manifold, as we discuss in Section \ref{sec:riemannian}. However, for the sake of this first investigation we will focus on a single special case - the $d$-dimensional flat cubical torus.

By `flat torus' we refer to the quotient $\T^d  = \R^d / \Z^d$. In other words, the flat torus can be thought of as the unit cube in $\R^d$, with its opposite sides ``glued" together (see Figure \ref{fig:flat_torus}). Consequently, the metric $\rho$ on $\T^d$, known as the \emph{toroidal} metric,  is given by
\[
\rho(x,y) = \min_{\Delta\in \Z^d} \norm{x-y+\Delta},\quad x,y\in [0,1]^d.
\]
The main advantage of the flat torus is that it allows us to work with a simple (almost) Euclidean metric, while avoiding boundary effects that exist, for example, in the cube.

\begin{figure}[h]
\centering
\includegraphics[scale=0.4]{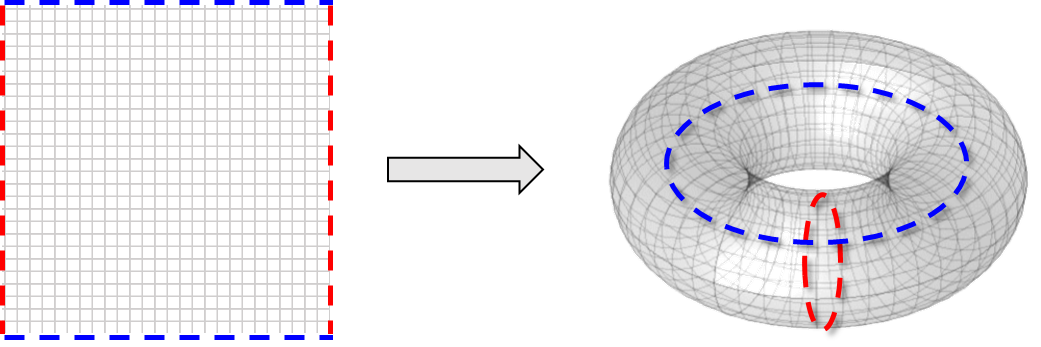}
\caption{\label{fig:flat_torus}  The flat torus $\T^2$ is obtained by taking the unit cube in $\R^2$ and  identifying the two pairs of opposite sides. This identification creates the two loops in the torus.}
\end{figure}

More specifically, the \emph{radius of convexity} $r_{\conv}$ of a Riemannian manifold $\cM$ is defined as the largest value $r$ such that every ball of radius $r$ in $\cM$ is convex. It is also the threshold below which any two points in a Riemannian ball can be connected by a unique short geodesic that lies entirely within the ball. In the case of flat torus we have $r_{\conv} = 1/2$, and it can be shown that every ball in $\T^d$ of radius less than $r_{\conv}$ can be isometrically embedded into $\R^d$. Since we will mostly be considering infinitesimally small neighborhoods, our calculations will be as if we are working in the Euclidean space.

Our choice to state and prove the results in this paper for the flat torus $(\T^d,\rho)$ stems from the relative elegant calculations involved using the toroidal metric. While locally, the torus $\T^d$ (as any other closed manifold) looks almost identical to the Euclidean space $\R^d$, some caution is required when considering global phenomena. Next, we wish to discuss the necessary adjustments needed to study the \cech complex and the distance function presented in Sections \ref{sec:cech} and \ref{sec:morse_dist}.


\subsection{The \cech complex on the torus}\label{sec:torus_cech}


A large part of our analysis will rely on the Nerve Lemma \ref{lem:nerve} (stating that $\cC(\cP,r)\simeq \cU(\cP,r)$). For this lemma to apply to $\cC(\cP,r)$ we need  the intersections of balls of radius $r$ to be  either empty or contractible. For balls in $\R^d$ this is always true, but on the torus (and other compact manifolds), this is not true for all radii. For example, on the torus, even a single ball of radius $1/2$ is not contractible (it covers the $1$-cycles of the torus). It is true however, for any radius smaller than $r_{\conv}$, and therefore we limit our discussion to the \cech complex $\cC(\cP,r)$ with $r< r_{\conv}$.
As we will see later, the vanishing of homology which we are seeking occurs with high probability at $r = o( r_{\conv})$, and therefore this restriction will not affect our results.

\subsection{The distance function on the torus}\label{sec:torus_dist_fn}

A similar caution is required studying the Morse theory for the distance function. In Section \ref{sec:morse_dist} we  defined the notion of critical points for the distance function $\rho_{_\cP}$ for $\cP\subset\R^d$ using the Morse theory for min-type functions in \cite{gershkovich_morse_1997}. We were able to do so due to the fact that $\rho^2(p,\cdot):\R^d\to\R$ is smooth, and therefore $\rho^2_{_\cP} = \min_{p\in\cP}\rho^2(p,\cdot)$ is a min-type function.
However, on the torus (as any compact Riemannian manifold) the situation is more delicate. If $p\in\T^d$, then $\rho^2(p,\cdot):\T^d\to\R$ is smooth at $x\in \T^d$ only if the geodesic from $x$ to $p$ is unique. Consequently, we have to be careful about how we apply Morse theory for min-type function in this case. We start with the following lemma.


\begin{lem}
Let $\cP\subset \T^d$ be a finite set. If $r < r_{\conv}$ then $\rho_{_\cP}^2:\cU(\cP,r)\to\R$ is a min-type function. 
\end{lem}


\begin{proof}
Although  $\rho^2(p,\cdot)$ is not  everywhere smooth in $\T^d$, it is smooth  in the ball  $B_r(p)$ for every $r < r_{\conv}$. 
Let  $x\in\T^d$ be such that $\rho_{_\cP}(x) < r_{\conv}$. Denoting $\cP_x := \set{p\in\cP  : \rho(p,x) < r_{\conv}}$ we can find 
 a small neighborhood $N_x$ such that $\rho(p,y) < r_{\conv}$ for every $p\in \cP_x$ and $y\in N_x$ and such that
\[
	\rho_{_\cP}(y) = \min_{p\in\cP_x} \rho(p,y),\quad \forall y\in N_x.
\]
 Thus, in $N_x$ we have that $\rho_{_\cP}^2$ is a min-type function (the minimum of smooth functions).
We conclude that if $r<r_{\conv}$ then in the set $\cU(\cP,r)$ the  function   $\rho_{_\cP}^2$  defines a germ of a min-type function.

\end{proof}

{\bf Remark:} This lemma allows us to study the homology of $\cU(\cP,r)$ for $r< r_{\conv}$ using the distance function, similar to $\R^d$. Note that the restriction $r<r_{\conv}$ is similar to the one that arose form the Nerve Lemma.


Another issue we need to resolve is the uniqueness of critical points.
If  $\cP\subset\R^d$ we saw in Lemma \ref{lem:gen_crit_point} that the critical point generated by a subset $\cY\subset\cP$ is unique. On the torus, this is not necessarily true, even if $r<r_{\conv}$. For example, the set $\cP = \set{(0.2,0), (0.8,0)}\subset\T^2$ generates two critical points of index $1$ (saddle points) at $c_1 = (0,0)$ and $c_2 = (0.5,0)$, with $\rho_{_\cP}(c_1) = 0.2$ and $\rho_{_\cP}(c_2) = 0.3$, both values are less than $r_{\conv}=0.5$. 
Requiring that $r<r_{\conv}/3$ resolves this issue as the following lemma states.


\begin{lem}
Let $r< r_{\conv}/3$. If $c\in \cU(\cP,r)$ is a critical point of index $k$, then it is generated by a subset $\cY\subset\cP$ of $k+1$ points (similarly to Definition \ref{def:crit_pts}). Furthermore, $c$ is the unique critical point in $\cU(\cP,r)$ generated by $\cY$.
\end{lem}


\begin{proof}

For every $x\in\cU(\cP,r)$ note that the values of $\rho_{_\cP}$ in the ball $B_r(x)$ are completely determined by points in the set $\cP\cap B_{3r}(x)$, i.e.~$\rho_{_\cP}|_{B_r(x)} = \rho_{_{\cP\cap B_{3r}(x)}}|_{B_r(x)}$.
If $r< r_{\conv}/3$ the ball $B_{3r}(x)$ can embedded isometrically in $\R^d$, which implies that $\rho_{_\cP}|_{B_r(x)}$ behaves identically to the distance function in $\R^d$ discussed in Section \ref{sec:morse_dist}. If $c\in \cU(\cP,r)$ is a critical point of $\rho_{_\cP}$, then using the isometry $B_{3r}(c)\to \R^d$, Definition \ref{def:crit_pts} applies to it as well.
In particular, following Lemma \ref{lem:gen_crit_point},  $c$ is generated by a subset $\cY\subset \cP$ of $k+1$ points, such that $\rho_{_\cP}(c) = \rho_{_\cY}(c) \le r$.  

Suppose that there exists another critical point $c' \in \cU(\cP,r)$ generated by $\cY$. Since $\rho_{_\cP}(c') = \rho_{_\cY}(c') \le r$ as well, we must have $c'\in B_{2r}(c)\subset B_{3r}(c)$. However, since in $\R^d$ the set $\cY$ generates a unique critical point, the same should hold for the embedding $B_{3r}(c)\to \R^d$, and therefore $c=c'$.

\end{proof}


{\bf Remark:} Using the uniqueness of critical points in $\cU(\cP,r)$ ($r< r_{\conv}/3$),  we can define $S(\cY)$, $C(\cY)$, $R(\cY)$, $B(\cY)$,  and $\Delta(\cY)$ similarly to \eqref{eq:def_S}-\eqref{eq:def_Delta} using the local isometry described above. 

To conclude, while the distance function on the torus $\T^d$  introduces some features that do not occur in $\R^d$ (e.g.~singular points, and non-uniqueness of critical points), if $r<r_{\conv}/3$ then in $\cU(\cP,r)\subset\T^d$ the critical points of the distance function behave the same was as in $\R^d$. Consequently, we define
\begin{equation}\label{eq:r_max}
r_{\max}:= r_{\conv}/3,
\end{equation}
which will be useful for us later.
As mention earlier, the vanishing threshold we look for is much smaller than $r_{\conv}$ (and $r_{\conv}/3$). Therefore, we will be able to analyze the distance function on the torus similarly to the Euclidean space.


\section{Definitions and notation}


\subsection{The Poisson process}\label{sec:poisson}


In this paper we study the homology of random \cech complex constructed from a random set of points. The points are generated by a homogeneous Poisson process which can be defined as follows. Let $(X, \mu)$ be a compact metric-measure space and let $X_1,X_2,\ldots$ be a sequence of $\iid$ random variables uniformly distributed in $X$. Let $N\sim\pois{n}$ be a Poisson random variable, independent of the $X_i$-s. Then we define
\begin{equation}\label{eq:def_pois}
\cP_n = \set{X_1,\ldots, X_N}.	
\end{equation}
This definition is equivalent to saying that $\cP_n$ is a homogeneous Poisson process with intensity $n$. In particular then the following holds:
\begin{enumerate}
\item For every compact measurable set $A$: $\abs{\cP_n\cap A}\sim \pois{n\mu(A)}$
\item If $A$ and $B$ are disjoint, then the variables ${\cP_n\cap A}$ and ${\cP_n\cap B}$ are independent\\ (this property is usually referred to as `spatial independence').
\end{enumerate}

We will be interested in asymptotic behavior as $n\to \infty$.
For any event $A$ that depends on $\cP_n$, we say that $A$ occurs \emph{with high probability} (w.h.p.) if $\limninf\prob{A} = 1$.


\subsection{Some notation}


The following objects will be used repeatedly in the paper:
\begin{itemize}
\item The union of balls - $\cU(n,r) := \cU(\cP_n,r)$,
\item The \cech complex - $\cC(n,r) := \cC(\cP_n,r)$,
\item The distance function - $\rho_n(\cdot) := \rho_{_{\cP_n}}(\cdot)$.
\end{itemize}

Throughout the paper, $k$ will mostly be a fixed positive integer value referring to either an index of a critical point or degree of homology. For a fixed $k$, we will use $\cY$ to  represent a set of $k+1$ random variables. When evaluating probabilities or moments using integrals will use the following notation -
\begin{itemize}
\item $x\in \T^d$ - a single variable,
\item $\bx = (x_0,\ldots, x_k) \in (\T^d)^{k+1}$ - a $(k+1)$-tuple of points in $\T^d$,
\item $\by = (y_1,\ldots, y_k)\in (\R^d)^k$ - a $k$-tuple of points in $\R^d$,
\item If $f:(\T^d)^{k+1}\to \R$ then $f(\bx) := f(x_0,\ldots, x_k)$,
\item If $f:(\R^d)^{k+1}\to \R$ then $f(0,\by) := f(0,y_1,\ldots, y_k)$.
\end{itemize}

In random \cech complexes, similarly to random geometric graphs, the term that controls much of the limiting behavior of the complex is
\begin{equation}\label{eq:def_lambda}
	\Lambda := \omega_d nr^d,
\end{equation}
where $\omega_d$ is the volume of a $d$-dimensional Euclidean unit ball.
The value $\Lambda$ is the expected number of points in a ball of radius $r$. This value is directly related to the expected vertex degree (number of neighbors) in the random \cech complex $\cC(n,r)$. The higher $\Lambda$ is, the denser the complex is. The vanishing thresholds we study in this papers are values of $\Lambda$ that guarantee the vanishing of homology.

Finally, we will use the following  asymptotic notation.
\begin{itemize}
\item $a(n)\approx b(n) \ \ \Rightarrow\ \ \limninf a(n)/b(n) = 1$, 
\item $a(n)\sim b(n)\ \ \Rightarrow \ \ \limninf a(n)/b(n) \in (0,\infty)$,  (i.e.~$a(n) = \Theta(b(n))$) ,
\item $a(n) \ll b(n)\ \ \Rightarrow\ \ \limninf a(n)/b(n) = 0$ (i.e.~$a(n) = o(b(n))$).
\end{itemize}


\section{Main results}\label{sec:main_results}


Let $\cP_n$ be the homogeneous Poisson process in $\T^d$ with intensity $n$, and $\cC(n,r)$ be the corresponding \cech complex. If $r$ is large enough we expect the union of balls $\cU(n,r)$ to cover the torus and then, by the Nerve Lemma \ref{lem:nerve}  we have $H_k(\cC(n,r)) \cong H_k(\T^d)$ for every $0\le k \le d$. 
Translating the results in  \cite{flatto_random_1977} to our case (see Corollary \ref{cor:coverage}) we have that the threshold value for coverage is 
\[
	\Lambda_c := \log n +(d-1)\log\log n.
\]
In addition, the analysis of connectivity in random geometric graphs (cf. \cite{penrose_random_2003}) yields that the threshold for connectivity of $\cC(n,r)$ is
\begin{equation}\label{eq:lambda_0}
	\Lambda_0 := \frac{\log n}{2^d}.
\end{equation}
In other words, $\Lambda_0$ is the threshold value above which we have $H_0(\cC(n,r)) \cong H_0(\T^d)$.
Analogously to other models of random simplicial complexes (e.g. \cite{kahle_sharp_2014,meshulam_homological_2009}), we expect that between $\Lambda_0$ and $\Lambda_c$ there would be an increasing sequence of threshold values 
\[
	\Lambda_0 < \Lambda_1 < \Lambda_2 < \ldots < \Lambda_{d} \le \Lambda_c,
\]
such that $\Lambda_k$ is the threshold value to have $H_k(\cC(n,r))\cong H_k(\T^d)$.
We refer to $\Lambda_k$ as the `vanishing threshold' for $H_k$, in the sense that above this value we see only the $k$-cycles that belong the torus, and all the other cycles that appeared in lower radii vanish.

We need to separate the top ($d$-dimensional) homology from the rest of the homology groups. Prior to coverage, since $\cU(n,r)$ is a ``nice" (compact and locally contractible) proper subset of $\T^d$ then necessarily $H_d(\cU(n,r)) = 0$ (cf. \cite{hatcher_algebraic_2002}). Once we cover the torus, we know that $\cU(n,r)= \T^d$, and therefore $H_d(\cU(n,r)) \cong H_d(\T^d)$.  Corollary \ref{cor:coverage} then immediately implies the following.


\begin{thm}\label{thm:top_homology}
Suppose that $w(n)\to \infty$ as $n\to\infty$.
\[
	\limninf\prob{H_d(\cC(n,r)) \cong H_d(\T^d)} = \begin{cases} 1 & \Lambda = \log n + (d-1)\log\log n + w(n) \\
	0 & \Lambda = \log n + (d-1)\log\log n - w(n). \end{cases}.
\]
\end{thm}


In other words, $\Lambda_d \approx \log n$, while $\Lambda_0 \approx 2^{-d}\log n$. 
Comparing $\Lambda_d$ and $\Lambda_0$ \eqref{eq:lambda_0} a reasonable conjecture would be that the other vanishing thresholds are of the form $\Lambda_k \approx c_k \log n$ for some $c_k \in (2^{-d},1)$.
The main theorems of this paper provide bounds for $\Lambda_k$ ($k=1,\ldots,d-1$) and show that this conjecture is actually wrong and in fact $\Lambda_k \approx \log n$ for all $1\le k\le d$.

To analyze the vanishing of the $k$-th homology ($1\le k \le d-1$), we start by bounding the expected value of $\beta_k(r)$ - the $k$-th Betti number of $\cC(n,r)$ (i.e.~the number of non-trivial $k$-cycles).


\begin{prop}\label{prop:betti_order}
If $r\to 0$ and $\Lambda\to\infty$, then for every $1\le k\le  d-1$ there exist constants $a_{k},b_{k}>0$ such that
\[
a_{k} n\Lambda^{k-2}e^{-\Lambda} \le	\meanx{\beta_k(r)} \le \beta_k(\T^d) + b_{k}{n\Lambda^k e^{-\Lambda}}.
\]
\end{prop}


Consequently, we have the following phase transition.

\begin{cor}\label{cor:betti_bounds}

Suppose that $w(n)\to \infty$ as $n\to\infty$ and let $1\le k \le d-1$.
Then
\[
	\limninf\meanx{\hat\beta_k(r)} =\begin{cases} \beta_k(\T^d) & \Lambda = \log n + k\log\log n + w(n) \\
	\infty & \Lambda = \log n + (k-2)\log\log n - w(n).
	\end{cases}
\]
\end{cor}


Finally, we would like to turn the statement about the expectation into a statement about the vanishing probability. This results in the following theorem.


\begin{thm}\label{thm:homology_bounds}
Suppose that $w(n)\to \infty$ as $n\to\infty$ and let $1\le k \le d-1$.
\[
\limninf\prob{H_k(\cC(n,r)\cong H_k(\T^d)} = \begin{cases} 1 & \Lambda = \log n + k\log\log n + w(n) \\
0 & \Lambda = \log n + (k-2)\log\log n - w(n),
\end{cases}
\]
where for the lower threshold we require further that $w(n)\gg \log\log \log n$.
\end{thm}


Theorem \ref{thm:homology_bounds} suggests  that the vanishing threshold for the $k$-th homology satisfies $$\Lambda_k \in [\log n + (k-2)\log\log n, \ \log n + k\log\log n].$$ 
In particular we have the following straightforward corollary.


\begin{cor}\label{cor:vanish_homology}
For any $1\le k \le d-1$, and for any $\eps \in(0,1)$, 
\[
\limninf\prob{H_k(\cC(n,r)\cong H_k(\T^d)} = \begin{cases} 1 & \Lambda = (1+\eps)\log n \\
0 & \Lambda = (1-\eps)\log n.
\end{cases}
\]

\end{cor}


This indicates that the conjecture made above about the vanishing thresholds is wrong. Except for $H_0$, all other homology groups $H_k$ vanish at the same scale $\Lambda \approx \log n$ (i.e. all the $c_k,\ k\ge 1$ coincide), and  the differences in the threshold values show up as a second order $\log \log n$ term. This result is highly counterintuitive, and we consider it one of the main contributions of this paper. In Section \ref{sec:conclusion} we will provide an intuitive explanation for this phenomenon.


We provide the proofs for the theorems above in Section \ref{sec:proofs}. The arguments required to prove the upper and lower bounds are different, but both rely heavily on Morse theory for the distance function $\rho_n$. 
In the following sections we provide the framework and Morse theoretic ideas that are needed to prove each of the bounds.


\section{Upper bound - Critical points}\label{sec:upper_bound}


The main idea in proving the upper bound in Proposition \ref{prop:betti_order} and Theorem  \ref{thm:homology_bounds} is the following. By Morse theory, changes to $H_k$ occur only at critical points of index $k$ and $k+1$.
We will show that if $\Lambda=\omega_dnr^d = \log n + k\log\log n + w(n)$ then there are no such critical points with value greater than $r$, and therefore the $H_k(\cC(n,r)) \cong H_k(\cC(n,r_{\max}))$. On the other hand, with high probability we have that $\cU(n,r_{\max})$ covers the torus, and therefore $H_k(\cC(n,r))\cong H_k(\T^d)$ (cf. Corollary \ref{cor:coverage}). In this section we provide the definitions and statements needed to make this argument rigorous.

The main object that will help us analyze the behavior of homology later is the following quantity - 
\begin{equation}\label{eq:C_k_r}
	C_k(r) := \# \set{\textrm{critical points $c\in \T^d$ of index $k$, with } \rho_n(c) \le r},
 \end{equation}
 i.e.\! the number of critical points of index $k$ with value in $[0,r]$.


\subsection{Expected value}


The following proposition provides the expected value of $C_k(r)$, as a function of $\Lambda$. Note that as opposed to most of the results related to random geometric graphs and complexes, the following statement is true for every $n$ and is \emph{not} asymptotic.


\begin{prop}\label{prop:mean_C_k_r} For $1\le k \le d$, and for every $r \le r_{\max}$  we have , 
\[
	\mean{C_k(r)} = D_k n \param{1-e^{-\Lambda} \sum_{j=0}^{k-1} \frac{\Lambda^j}{j!}},
\]
where $D_k$ is a constant that depends on $k$ and $d$ and is given in \eqref{eq:D_k},  $\Lambda$ is defined in \eqref{eq:def_lambda}, and $r_{\max}$ is defined in \eqref{eq:r_max}.
\end{prop}


An immediate corollary of Proposition \ref{prop:mean_C_k_r} is related to the Euler characteristic of the \cech complex.
The Euler characteristic is an integer valued topological invariant that can be viewed as a ``summary" of the Betti numbers. There are various equivalent definitions for the Euler characteristic of a simplicial complex $\Sigma$. 
One of the definitions,  via Betti numbers, is
\begin{equation}\label{eq:ec_betti}
    \chi(\Sigma) = \sum_{k=0}^{\infty} (-1)^k \beta_k(\Sigma).
\end{equation}
In other words, the Euler characteristic is a (signed) sum of the number of cycles of all possible dimensions. In the special case where $\Sigma = \cC(n,r)$, defining
\[
	\chi(r) := \chi(\cC(n,r)),
\]
we can use Morse theory (see \cite{milnor_morse_1963}) to show that
\[
	\chi(r) = \sum_{k=0}^d (-1)^k C_k(r).
\] 
Thus, using Proposition \ref{prop:mean_C_k_r} we have the following.


\begin{cor}\label{cor:mean_ec}
If $r\le r_{\max}$ then
\[
\mean{\chi(r)} = n e^{-\Lambda} \param{1+\sum_{j=1}^{d-1} A_j \Lambda^j}.
\]
for some constants $A_j$ that depend on $j$ and $d$, defined in \eqref{eq:A_j} below.
\end{cor}


{\bf{Remarks:}} 
The expected Euler characteristic formula in Corollary \ref{cor:mean_ec} will not be used directly in proving the main results of this paper. Nevertheless, we included it here both for completeness, and for the fact that it provides valuable intuition which sometimes may even lead to concrete statements. At this point we want to mention the following observations.

\begin{enumerate}
\item The terms $A_jne^{-\Lambda}\Lambda^j$ appearing in the sum above vanish at $\Lambda = \log n + j\log\log n$. This implies that a significant change in the topology of the complex should occur at these points. While this intuitive statement is somewhat vague, it does agree with and strengthens the main results of this paper. 
\item
When $\Lambda \to \infty$ we have
\begin{equation}\label{eq:mean_ec_super}
	\mean{\chi(r)} \approx A_{d-1} n\Lambda^{d-1}e^{-\Lambda},
\end{equation}
where $A_{d-1} = \frac{1}{(d-1)!}(-1)^{d-1}D_d$, and $D_d > 0$. Thus, we can conclude that for very large $\Lambda$ the Euler characteristic converges to $\chi(\T^d)=0$, and its sign depends on the dimension of the torus. The intuition behind this phenomenon is that as we get  closer to  covering the torus  two main things happen:
\begin{itemize}
\item All the nontrivial cycles in the homology of $\T^d$ have been formed, except for the  cycle in the top dimensional homology.
\item The lack of coverage introduces many small voids (``bubbles") generating nontrivial $(d-1)$-cycles. 
\end{itemize} 
If $K$ is the number of bubbles, the Euler characteristic would therefore be 
\[
\chi(r) \approx \chi(\T^d) - (-1)^d + K(-1)^{d-1} = (K+1)(-1)^{d-1},
\]
which explains why $\sign(\chi(r)) \approx (-1)^{d-1}$ in this case.
\end{enumerate}


\subsection{Vanishing thresholds for critical points}


Recall the definition of $C_k(r)$ in \eqref{eq:C_k_r}, and define
\[
\hat C_k(r) := C_k(r_{\max})-C_k(r),
\]
i.e.~the number of critical points $c$ with $\rho_{n}(c)\in (r,r_{\max}]$. 
Since, with high probability we do not expect to have critical points with values higher than $r_{\max}$, the value of $\hat C_k(r)$ can be thought of as the ``remaining" number of critical points.
Using Proposition \ref{prop:mean_C_k_r} we have that 
\begin{equation}\label{eq:mean_C_k_diff}
	\meanx{\hat C_k(r)}= D_{k} n \param{e^{-\Lambda} \sum_{j=0}^{k-1} \frac{\Lambda^j}{j!}-e^{-\Lambda_{\max}} \sum_{j=0}^{k-1} \frac{\Lambda_{\max}^j}{j!}},
	\end{equation}
where  $\Lambda = \omega_d nr^d$, and $\Lambda_{\max} = \omega_d nr_{\max}^d $, which implies that $\Lambda \ll \Lambda_{\max}$ (since $r\to 0$ and $r_{\max}$ is positive).
If, in addition, we assume that $\Lambda\to\infty$, the dominant part in \eqref{eq:mean_C_k_diff} is $D_{k} n e^{-\Lambda}\Lambda^{k-1}$, and therefore we have the following phase transition.


\begin{cor}\label{cor:vanish_C_k}
Suppose that $\Lambda=  \log n + (k-1)\log\log n + w(n)$, then
\[
	\limninf\meanx{\hat C_k(r)} = \begin{cases} \infty & w(n) \to -\infty,\\
	D_k e^{-a} & w(n)\to a\in(-\infty,\infty), \\
	0 & w(n) \to \infty.\end{cases}
\]
\end{cor}


The above corollary suggests that $\Lambda = \log n + (k-1)\log\log n$ is the threshold for the vanishing of $\hat C_k(r)$, i.e.~it is a point above which we do not have any more critical points of index $k$. 
As mentioned in the beginning of this section, the vanishing of the critical points combined with Morse theory will help us  prove the upper bound part of Proposition \ref{prop:betti_order} and Theorem \ref{thm:homology_bounds}.


\section{Lower bound - $\Theta$-cycles}\label{sec:theta_cycles}


To prove the lower bound part in Proposition \ref{prop:betti_order} and Theorem \ref{thm:homology_bounds} we will show that there exist nontrivial $k$-cycles in $H_k(\cC(n,r))$ that do not belong to the torus $\T^d$ (i.e.~they are mapped to trivial cycles in $H_k(\T^d)$). 
To prove that such cycles exist, we will be studying a special construction of $k$-cycles that shows up when $\Lambda\to\infty$, and is localized around critical points of index $k$. 

Let $\cP\subset\R^d$, and $\cY\subset\cP$. 
Recall that if $\cY$ generates a critical point of the distance function $\rho_{_\cP}$, then $C(\cY)$ is the critical point itself , $R(\cY)$ is the critical value (the distance), and $\Delta(\cY)$ is the embedded simplex spanned by $\cY$. For a fixed $\eps \in (0,1)$  define
\begin{equation}\label{eq:A_eps}
	A_\eps(\cY) := B_{R(\cY)}(C(\cY)) \backslash B^\circ_{\eps R(\cY)}(C(\cY)),
\end{equation}
where $B_r(x)$ is a closed ball of radius $r$ centered at $x$, and $B^\circ_r(x)$ is an open ball, i.e.~$A_\eps(\cY)$ is a closed annulus around $C(\cY)$. Recall that by Morse theory, every index $k$ critical point either creates a new nontrivial $k$-cycle or terminates an existing $(k-1)$-cycle. The following Lemma provides a sufficient condition for the former to happen. 
While the statement is phrased for point sets in $\R^d$, using the local isometry of the torus into $\R^d$ , we can show that this is true for critical points $c\in\T^d$ with $\rho_{_\cP}(c) \le r_{\max}$.


\begin{lem}[$\Theta$-cycles]\label{lem:theta_cycle}
Let $\cP\subset\R^d$ be a set of points, $\rho_{_\cP}$ the distance function, and $\cY\subset\cP$ be a  set of $k+1$ points  generating a critical point of index $k$. 
Set
\begin{equation}\label{eq:epsilon}
\phi = \phi(\cY) := \frac{\inf_{x\in\partial \Delta(\cY)}\norm{x-C(\cY)}}{2R(\cY)},
\end{equation}
where $\partial \Delta(\cY)$ is the union of the faces on the boundary of the simplex $\Delta(\cY)$.
If  $A_\phi(\cY) \subset \cU(\cP,  R(\cY))$,
then the critical point generated by $\cY$ creates a new nontrivial cycle in $H_k(\cC(\cP, R(\cY)))$. A cycle created this way will be called a ``$\Theta$-cycle".
\end{lem}


{\bf{Remark:} }
Note that for $k=1$, $\partial\Delta(\cY) = \cY$ (the boundary of a $1$-simplex are its end points). In this case we have $\phi(\cY) = 1$ and then $A_\phi(\cY) = \partial B(\cY)$ (i.e.~a $(d-1)$-sphere). The statement above is still true.


The proof of this lemma (see Section \ref{sec:proofs}) requires some more working knowledge of algebraic topology, than is required for the rest of this paper. Therefore, we wish to provide the following intuitive explanation. 
At a critical point of index $k$ a new $k$-dimensional simplex is added to the complex $\cC(n,r)$, which is isolated in the sense that it is not a face of any higher dimensional simplex. We argue that adding this face necessarily generates a new $k$-cycle.
Since $A_\phi(\cY)$ is covered by $\cU(\cP,R(\cY))$, the $(d-1)$-sphere $\partial B(\cY)$ is covered as well, and so the new cycle is formed somewhere inside this sphere.
The effect of adding the new $k$-simplex can be viewed as ``splitting" the void inside the sphere using a $k$-dimensional plane. 
This action generates a nontrivial $k$-cycle. For example,  a line crossing a $2$-sphere will generate a $1$-cycle (or a loop), whereas adding a plane will split the sphere into two ``chambers" forming a $2$-cycle.
Figure \ref{fig:theta_cycles} provides  visual intuition to this explanation. Note that in the case where $d=2$ and $k=1$, we place an edge splitting the circle into two (creating a new $1$-cycle), and thus forming a shape similar to the letter $\Theta$, hence the name ``$\Theta$-cycles". 

To prove that $H_k(\cC(n,r))\ne H_k(\T^d)$ we will show that $\beta_k(r) > \beta_k(\T^d)$ by counting a subset of the cycles originating from $\Theta$-cycles.
For a fixed $\eps>0$, we define $\beta_k^\eps(r)$ to be the number of $\Theta$-cycles generated by subsets $\cY\subset\cP_n$ such that
\[
 (1)\ R(\cY) \in (r',r],\quad (2)\ B_{r''}(C(\cY))\cap\cP_n = \cY,\quad and\quad (3)\ \phi(\cY)\ge \eps,
\]
where
\[
	\quad r' = r(1-\delta), \quad r'' = r(1+\sqrt{2\delta}), \quad \delta = \Lambda^{-2}.
\]

Condition (1) ensures that the $\Theta$-cycles are created at a critical radius in $(r',r]$. Condition (2) ensures that at radius $r$, the new $k$-simplex added is still isolated, and therefore the new cycle it generates is nontrivial (this is justified in the proof of the next Lemma \ref{lem:theta_cycle_bound}). Condition (3) is required in order to be at a safe distance from the singular case  $\phi(\cY)=0$.
In other words, $\beta_k^\eps(r)$ counts $\Theta$-cycles that are created \emph{before} $r$ and terminated \emph{after} $r$, as stated in the following lemma.


\begin{lem}\label{lem:theta_cycle_bound}
Let $\eps>0$ be fixed, then
\[
\beta_k(r) \ge \beta_k^\eps(r).
\]
\end{lem}


Next, we will evaluate the expected value of $\beta_k^\eps(r)$.


\begin{lem}\label{lem:theta_cycle_order}
There exists $\eps>0$ such that if $\Lambda\to\infty$, then
\[
	\mean{\beta_k^\eps(r)} \sim n\Lambda^{k-2}e^{-\Lambda}.
\]
\end{lem}


This lemma provides the proof for the lower bound in Proposition \ref{prop:betti_order}. To prove the lower bound part of Theorem \ref{thm:homology_bounds} we will need to prove the following lemma.


\begin{lem}\label{lem:theta_cycle_vanish}
There exists $\eps>0$ such that for any $1\le k \le d-1$, if $\Lambda = \log n + (k-2)\log\log n -w(n)$ and $w(n)\gg\log\log\log n$ then
\[
\limninf\prob{\beta_k^\eps(r) > 0} = 1.
\]
\end{lem}


\begin{figure}[h!]
\centering
\subfigure[]
{
  \centering
     {\includegraphics[scale=0.35]{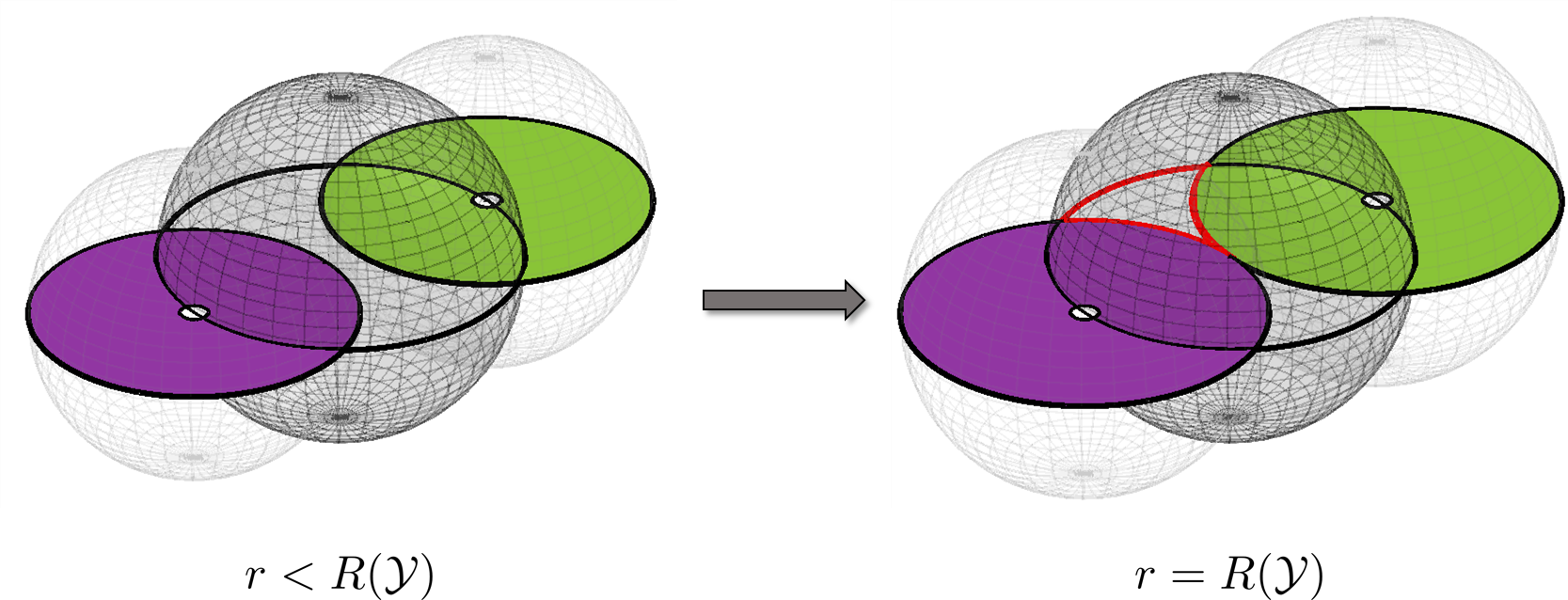}}
}
\subfigure[]
{
  \centering
   {\includegraphics[scale=0.35]{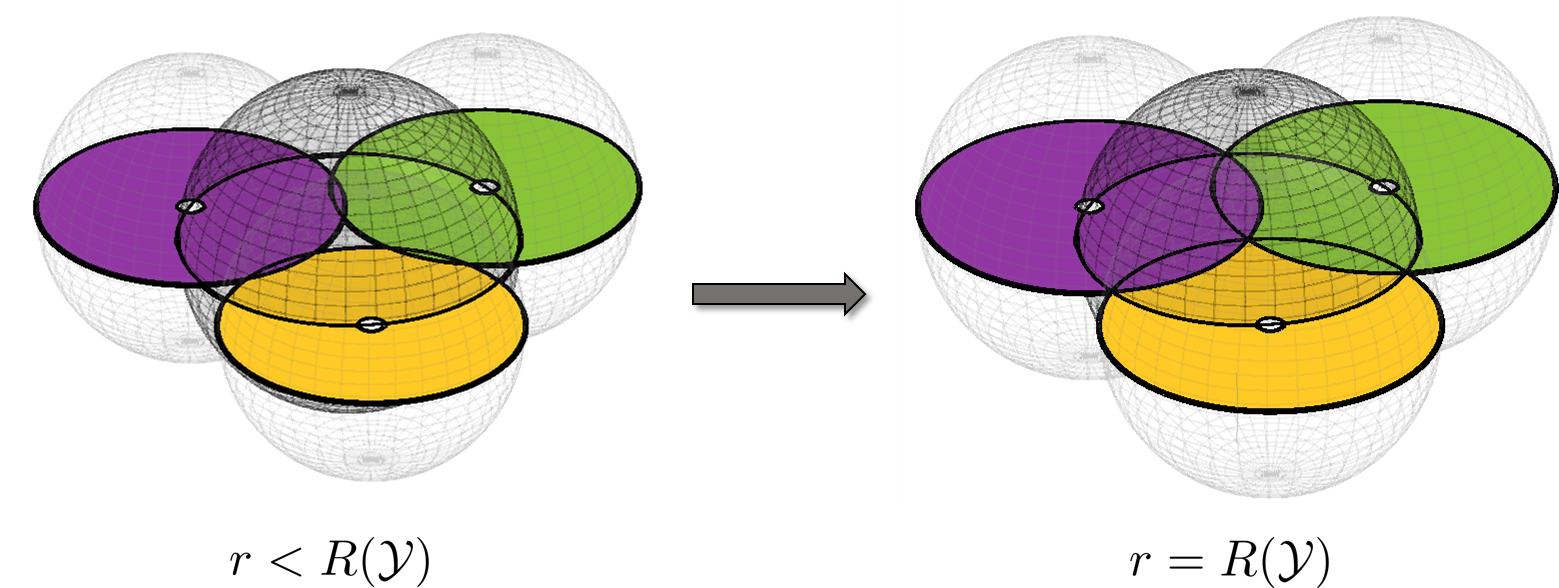}}
}

\caption{ A visual intuition for $\Theta$-cycles for $d=3$. (a) A $\Theta$-cycle with $k=1$: The set $\cY$ consists of the two white dots placed on the sphere $\partial B(\cY)$ of radius $R(\cY)$. 
We drew the balls around the two points as semi-transparent and emphasized the disc at the equator to demonstrate the effect of the critical point. Recall that the assumptions in Lemma \ref{lem:theta_cycle_vanish} imply that the sphere $\partial B(\cY)$ is covered by $\cU(\cP,r)$. Therefore, once the two balls touch (at $r=R(\cY)$) the gap between the balls and the covered sphere creates a closed loop, or a $1$-cycle (marked in red). This effect is equivalent placing a $1$-dimensional ball (a string) at the equator. (b) A $\Theta$-cycle with $k=2$: In this case the set $\cY$ consists of three points. Once the balls around the points touch, they divide the void inside the sphere $\partial B(\cY)$ into two, introducing a new $2$-cycle. Topologically, this is equivalent to placing a $2$-dimensional disc at the equator.
}
\label{fig:theta_cycles}
\end{figure}


\section{Proofs}\label{sec:proofs}


In this section we wish to present the detailed proofs for the statements in Sections \ref{sec:main_results}-\ref{sec:theta_cycles}. 
In Section \ref{sec:proofs_5} we prove the main results that appeared in Section \ref{sec:main_results}, based on the lemmas that appeared later in Sections \ref{sec:upper_bound} and \ref{sec:theta_cycles}. Next, in Sections \ref{sec:proofs_6} and \ref{sec:proofs_7} we prove the lemmas that appeared in Sections \ref{sec:upper_bound} and \ref{sec:theta_cycles}, respectively.


\subsection{Proofs for Section \ref{sec:main_results}} \label{sec:proofs_5}


\begin{proof}[Proof of Proposition \ref{prop:betti_order}] \ \\
We need to show that $
a_{k} n\Lambda^{k-2}e^{-\Lambda} \le	\meanx{\beta_k(r)} \le \beta_k(\T^d) + b_{k}{n\Lambda^k e^{-\Lambda}}$.

{\bf Upper bound:} Denote $\hat\beta_k(r) = \beta_k(r)-\beta_k(\T^d)$, then we need to show that $\meanx{\hat\beta_k(r)}\le b_k n\Lambda^k e^{-\Lambda}$ for some $b_k>0$. This will be done using the expected number of critical points provided by Proposition \ref{prop:mean_C_k_r}. The main idea is that the number of critical points of index $k+1$ occurring at a radius greater than $r$ (i.e.~$\hat C_{k+1}(r)$) should serve as an upper bound to $\hat \beta_k(r)$, since all the extra cycles (the ones that do not belong to the torus) must be terminated at some point.

Let $E$ denote the event that $\cU(n,r_{\max})$ covers the torus. Then,
\begin{equation}\label{eq:mean_beta_k_cond}
	\meanx{\hat\beta_k(r)} = \cmeanx{\hat\beta_k(r)}{E}\prob{E} + \cmeanx{\hat\beta_k(r)}{E^c}\prob{E^c}.
\end{equation}
If the event $E$ occurs, then $\rho_n(x)\le r_{\max}$ everywhere on the tours, and therefore using Morse theory (see Section \ref{sec:torus_dist_fn}) every extraneous $k$-cycle in $\cC(n,r)$ that does not belong to the torus will eventually be terminated by a critical point of index $k+1$ with value  in $(r,r_{\max}]$.
 Therefore, we must have $\hat\beta_k(r) \le  \hat C_{k+1}(r) = C_{k+1}(r_{\max}) - C_{k+1}(r)$, and then
\begin{equation}\label{eq:bound_mean_betak_cover}
\cmeanx{\hat\beta_k(r)}{E}\prob{E} \le  \cmeanx{\hat C_{k+1}(r)}{E}\prob{E} 
\le\meanx{\hat C_{k+1}(r)}.
\end{equation}
Using Proposition \ref{prop:mean_C_k_r} we have that 
\begin{equation}
	\meanx{\hat C_{k+1}(r)}= D_{k+1} n \param{e^{-\Lambda} \sum_{j=0}^{k} \frac{\Lambda^j}{j!}-e^{-\Lambda_{\max}} \sum_{j=0}^{k} \frac{\Lambda_{\max}^j}{j!}},
	\end{equation}
where  $\Lambda = \omega_d nr^d$, and $\Lambda_{\max} = \omega_d nr_{\max}^d $, which implies that $\Lambda \ll \Lambda_{\max}$ (since $r\to 0$ and $r_{\max}$ is positive). Since, in addition, we assume that $\Lambda\to\infty$, the dominant part in \eqref{eq:mean_C_k_diff} is $D_{k+1} n e^{-\Lambda}\Lambda^{k}$.
Putting it back into \eqref{eq:bound_mean_betak_cover} we have
\[
\cmeanx{\hat\beta_k(r)}{E}\prob{E} = O(n e^{-\Lambda}\Lambda^{k}).
\]
To complete the proof we show that the second summand in \eqref{eq:mean_beta_k_cond} is $o( n e^{-\Lambda}\Lambda^k)$. For any simplicial complex, the $k$-th Betti number is bounded by the number of $k$-dimensional faces (see, for example, \cite{hatcher_algebraic_2002}). Since the number of faces is bounded by $\binom{\abs{\cP_n}}{k+1}$, we have that 
\begin{equation}\label{eq:mean_betak_no_cover}\begin{split}
	\cmeanx{\beta_k(r)}{E^c} \prob{E^c}&\le \cmean{\binom{\cP_n}{k+1}}{E^c}\prob{E^c} \\
	&= \sum_{m=k+1}^\infty \binom{m}{k+1}\cprob{\abs{\cP_n} = m}{E^c}\prob{E^c} \\
	&=\sum_{m=k+1}^\infty \binom{m}{k+1}\cprob{E^c}{\abs{\cP_n}=m}\prob{\abs{\cP_n} = m},
	\end{split}
\end{equation}
where we used Bayes' Theorem. Since $\cP_n$ is a homogeneous Poisson process with intensity $n$ we have that $\prob{\abs{\cP_n}=m} = \frac{e^{-n}n^m}{m!}$, and also that given $\abs{\cP_n}=m$ we can write $\cP_n$ as a set of $m$ $\iid$ random variables $\cX_m = \set{X_1,\ldots, X_m}$ uniformly distributed on the torus. Therefore,
\[
	\cprob{E^c}{\abs{\cP_n}=m} = \prob{\T^d \not\subset \cU(\cX_m,r_{\max})}.
\]
Let $\eps= r_{\max}/2$, and let $\cS$ be a $\eps$-net of $\T^d$, i.e.~for every $x\in\T^d$ there is a point $s\in \cS$ such that $\rho(s,x)\le \eps$.
Since $r_{\max}$ is fixed, we can find such an $\eps$-net with $\abs{\cS} = M$, where $M$ is a constant that depends only on $d$. Note that if for every $s\in\cS$ there is  $X_i\in \cX_m$ with $\rho(s,X_i)\le \eps$, then for every $x\in \T^d$ we have
\[
\rho(x,X_i) \le \rho(x,s) + \rho(s,X_i) \le 2\eps =  r_{\max},
\]
and therefore $\T^d \subset \cU(\cX_m,r_{\max})$. Thus, if $\T^d\not\subset\cU(\cX_m,r_{\max})$, then there exists $s\in \cS$ with $\rho_{_{\cX_m}}(s) > \eps$, which yields
\[
\prob{\T^d \not\subset \cU(\cX_m,r_{\max})}\le \sum_{s\in\cS}\prob{\rho_{\cX_m}(s) >\eps} = M(1-\omega_d \eps^d)^m .
\]
Putting it all back into \eqref{eq:mean_betak_no_cover} we have
\[
	\cmeanx{\beta_k(r)}{E^c} \prob{E^c} \le \sum_{m=k+1}^\infty \binom{m}{k+1} M(1-\omega_d\eps^d)^m \frac{e^{-n}{n^m}}{m!} = \frac{M}{(k+1)!}(n(1-\omega_d\eps^d))^{k+1} e^{-n\omega_d\eps^d}.
\]
Since $\eps>0$ is a fixed constant, and since $\Lambda \ll n$ and $\Lambda\to\infty$, we have that
\[
\cmeanx{\beta_k(r)}{E^c} \prob{E^c}=  o(n e^{-\Lambda}\Lambda^k).
\]
Since $\hat\beta_k(r)\le \beta_k(r)$ we also have
\[
\cmeanx{\hat\beta_k(r)}{E^c} \prob{E^c}=  o(n e^{-\Lambda}\Lambda^k).
\]
To conclude, we have $\meanx{\hat\beta_k(r)} = O(n\Lambda^ke^{-\Lambda})$.
That completes the proof for the upper bound.
\\


{\bf Lower bound:} We need to show that $\meanx{\beta_k(r)} \ge a_kn\Lambda^{k-2}e^{-\Lambda}$ for some $a_k > 0$.  From Lemma \ref{lem:theta_cycle_order} we know that $\mean{\beta_k^\eps(r)}\sim n\Lambda^{k-2}e^{-\Lambda}$. From  Lemma \ref{lem:theta_cycle_bound}  we have that $\beta_k(r) \ge \beta_k^\eps(r)$. That completes the proof.
\end{proof}


Next, we prove Corollary \ref{cor:betti_bounds}.
\begin{proof}[Proof of Corollary \ref{cor:betti_bounds}]\ \\
From Proposition \ref{prop:betti_order} we have that $\meanx{\beta_k(r)} \le \beta_k(\T^d)+ b_k ne^{-\Lambda}\Lambda^{k} $ and therefore if $\Lambda = \log n + k\log\log n +w(n)$ we have $\meanx{\hat\beta_k(r)} \to \beta_k(\T^d)$.
In addition, we also have that $\meanx{\beta_k(r)}\ge a_k n\Lambda^{k-2}e^{-\Lambda}$.
Taking $\Lambda = \log n + (k-2)\log\log n -w(n)$ we then have that $\meanx{\beta_k(r)} \to \infty$.
\end{proof}


Finally, we will prove Theorem \ref{thm:homology_bounds}.


\begin{proof}[Proof of Theorem \ref{thm:homology_bounds}]\ 


{\bf Upper bound:} 
From Corollary \ref{cor:coverage}  for every $\delta \in (r_{\max}/2,r_{\max})$ we have that $\T^d \subset \cU(n,\delta)$ with high probability, implying that $H_k(\cC(n,\delta))\cong H_k(\T^d)$.
From Corollary \ref{cor:vanish_C_k} we know that if $\Lambda = \log n + k\log\log n  + w(n)$ then  both $\meanx{\hat C_{k}(r)} \to 0$ and $\meanx{\hat C_{k+1}(r)} \to 0$.
Using Markov's inequality, this implies that with high probability both $\hat C_k(r) = 0$ and $\hat C_{k+1}(r) = 0$.
By Morse Theory, this implies that for  $\delta \in (r,r_{\max})$ the $k$-th homology of $\cC(n,\delta)$ does not change. Therefore, we conclude that with high probability $H_k(\cC(n,r)) \cong H_k(\T^d)$.

{\bf Lower bound:} If $\Lambda = \log n + (k-2)\log\log n - w(n)$, and $w(n) \gg \log\log \log n$ then from Lemma \ref{lem:theta_cycle_vanish} we have
$
	\prob{\beta_k^\eps(r)>0}\to 1
$
for every  $1\le k\le d-1$. From Lemma \ref{lem:theta_cycle_bound} we have that $\beta_k(r) \ge \beta_k^\eps(r)$. Therefore, with high probability we have $\beta_k(r) >0$ which completes the proof.

\end{proof}


\subsection{Proofs for Section \ref{sec:upper_bound}} \label{sec:proofs_6}


\begin{proof}[Proof of Proposition \ref{prop:mean_C_k_r}] \ \\ 
We want to show that $\mean{C_k(r)} = D_k n \param{1-e^{-\Lambda} \sum_{j=0}^{k-1} \frac{\Lambda^j}{j!}}$ for $r\le r_{\max}$.
First, note that the expected number of critical points with $r=r_{\max}$ is zero, since 
this requires a subset of $k+1$ points to lie on a sphere of radius exactly $r_{\max}$ and this happens with probability zero. Thus, for the remainder of the proof we will assume that $r<r_{\max}$.

Let $\cP \subset \T^d$ be a finite set of points with its corresponding distance function $\rho_{_\cP}$, and let $c$ be a critical point of the distance function $\rho_{_\cP}$ with $\rho_{_\cP}(c)\le r$.  Following the discussion in Section \ref{sec:torus_dist_fn} and Lemma \ref{lem:gen_crit_point}, since $r< r_{\max}$ then $c$ is generated by a subset  $\cY$ that satisfies the following:
\begin{equation}\label{eq:crit_pt_cond}
\textrm{(1) } C(\cY) \in \Delta(\cY),\qquad \textrm{(2) } B(\cY) \cap \cP = \emptyset,\qquad \textrm{(3) } R(\cY) \le r,
\end{equation}
where $C(\cY), R(\cY), B(\cY)$ and $\Delta(\cY)$ are defined in \eqref{eq:def_C}-\eqref{eq:def_Delta}.
For any $\bx\in (\R^d)^{k+1}$ and  $\cP\subset\R^d$, we define the following indicator functions:
\begin{equation}\label{eq:h_funcs}
\begin{split}
	h(\bx) &:= \ind\set{ C(\bx) \in \Delta(\bx)},\\
	h_r(\bx) &:= h(\bx)\ind\set{R(\bx) \le r},\textrm{ and }\\
	g_r(\bx, \cP) &:= h_r(\bx)\ind\set{B(\bx)\cap \cP = \emptyset}.
\end{split}
\end{equation}
These functions can also be applied to points on the torus, provided that $\bx$ is contained in a ball of radius less than $r_{\max}$ and using the isometry into $\R^d$. If this is the case, we will define $h(\bx), h_r(\bx),$ and $g_r(\bx)$ the same way, otherwise we will set the function value  to be zero.
Using the above definition we can write
\[
	C_k(r) = \sum_{\substack{\cY\subset\cP_n\\\abs{\cY} = k+1}} g_r(\cY,\cP_n).
\]
Applying Palm theory to the mean value (see Theorem \ref{thm:prelim:palm}) we have
\[
\mean{C_k(r)} = \frac{n^{k+1}}{(k+1)!} \mean{g_r(\cY',\cY'\cup\cP_n)},
\]
where $\cY'$ is a set of $(k+1)$ $\iid$ points on $\T^d$, independent of $\cP_n$. Note that by the properties of the Poisson process $\cP_n$ we have that
\[
	\cmean{\ind\set{B(\cY')\cap \cP_n = \emptyset}}{\cY'} = \cprob{\abs{\cP_n\cap B(\cY')}=0}{\cY'}= e^{-n\vol(B(\cY'))} .
\]
Since we require that $R(\cY') \le r < r_{\max}$ then $B(\cY')$ is isometric to a Euclidean ball, and  $\vol(B(\cY')) = \omega_dR^d(\cY')$.
Therefore, 
\begin{equation}\label{eq:C_k_mean_start}
\mean{C_k(r)} = \frac{n^{k+1}}{(k+1)!}\int_{(\T^d)^{k+1}} h_{r}(\bx) e^{-n\omega_dR^d(\bx)} d\bx.
\end{equation}
Note that the functions $h_r(\bx)$ and $R(\bx)$ are translation invariant, and  therefore we can write
\[
\mean{C_k(r)} = \frac{n^{k+1}}{(k+1)!}\int_{(\T^d)^{k}} h_{r}(0,\bx') e^{-n\omega_dR^d(0,\bx')} d\bx',
\]
where $\bx' = (x_1,\ldots,x_k)$. Since the set $(0,\bx')$ is  contained in a ball of radius $r$, then necessarily we have that $x_i \in B_{2r}(0)$ for all $1\le i \le k$ and since $r< r_{\max}$ we can isometrically embed  this ball into $\R^d$. Thus,  from here on we can assume that we are integrating over points in $\R^d$ and write
\[
\mean{C_k(r)} = \frac{n^{k+1}}{(k+1)!}\int_{(\R^d)^{k}} h_{r}(0,\bx') e^{-n\omega_dR^d(0,\bx')} d\bx'.
\]

To proceed, we use the following change of variables:
\[
\begin{split}
	x_i &\to ry_i,\ \  1\le i \le k, \ \ y_i\in \R^d.
\end{split}
\]
Denoting $\by = (y_1,\ldots, y_k) \in (\R^d)^k$ we then have,
\[
	\begin{split}
	d\bx' &\to r^{dk} d\by, \\
	h_r(0,\bx') &\to h_r(0,r\by) = h_1(0,\by), \\
	R(0,\bx') &\to R(0,r\by) = r R(0,\by).
	\end{split}
\]
Therefore,
\begin{equation}\label{eq:N_k_r_1}
\begin{split}
	\mean{C_k(r)} &= \frac{n^{k+1}r^{dk}}{(k+1)!} \int_{(\R^d)^k} h_1(0,\by) e^{-\omega_d n r^d R^d(0,\by)} d\by  \\
	&= \frac{n \Lambda^k}{\omega_d^k(k+1)!}  \int_{(\R^d)^k} h_1(0,\by) e^{-\Lambda R^d(0,\by)} d\by.
\end{split}
\end{equation}
To further simplify this expression, we take the following change of variables. Instead of using the coordinates $y_1, \ldots,  y_k\in \R^d$, we use the fact that the $k+1$ points - $\{0,y_1,\ldots, y_k\}$ lie on the unique $(k-1)$-sphere $S(0,\by)$ (defined in \eqref{eq:def_S}). We can use spherical coordinates $(\tau, \theta)$ to identify the center of the sphere, where $\tau \in (0,\infty)$, and $\theta \in S^{d-1}$ (the $(d-1)$-unit sphere).  Since $0\in S(0,\by)$ we have that $R(0,\by) = \tau$. Once we know where the center is, in order to uniquely identify the sphere $S(0,\by)$, we need to choose the $k$-dimensional subspace in which it resides. We represent this choice by the variable $\gamma\in \Gamma_{d,k}$ - an element in the $k$-dimensional Grassmannian (the space of all $k$-dimensional linear subspaces of $\R^d$). Finally, after identifying the sphere $S(0,\by)$, we need to choose $k$ points on it, we do that by taking spherical coordinates -  $\theta_1,\ldots,\theta_k \in S^{k-1}$. Combined together we have the change of variables
\begin{equation}\label{eq:sphere_change}
	y_i \to \tau\cdot (\theta + \gamma \circ \theta_i),\quad 1\le i \le k.
\end{equation}
Defining $J(\tau,\theta, \gamma, \bth) := \frac{\partial \by}{\partial(\tau,\theta,\gamma,\bth)}$ we can rewrite \eqref{eq:N_k_r_1}  as
\[
	\mean{C_k(r)} = \frac{n\Lambda^{k}}{\omega_d^k(k+1)!}  \int_0^{1} \int_{\theta,\gamma,\bth} h(0,\theta+\gamma\circ \bth) e^{-\Lambda \tau^d}  J(\tau,\theta,\gamma,\bth)  d\theta d\gamma d\bth d\tau.
\]
Notice that \eqref{eq:sphere_change} implies that $J(\tau,\theta,\gamma,\bth) = \tau^{dk-1} J(1,\theta,\gamma,\bth)$, and therefore
\begin{equation}\label{eq:mean_nk_2}
\begin{split}
	\mean{C_k(r)} &= \frac{n\Lambda^{k}}{\omega_d^k(k+1)!}  \int_0^{1}  \tau^{dk-1} e^{-\Lambda \tau^d} d\tau \int_{\theta,\gamma,\bth} h(0,\theta+\gamma\circ \bth) J(1,\theta,\gamma,\bth)  d\theta d\gamma d\bth \\
	&= \tilde D_k n\Lambda^{k} \int_0^{1}  \tau^{dk-1} e^{-\Lambda \tau^d} d\tau,
\end{split}
\end{equation}
where $\tilde D_k:=\frac{1}{\omega_d^k(k+1)!}\int_{\theta,\gamma,\bth} h(0,\theta+\gamma\circ \bth) J(1,\theta,\gamma,\bth)  d\theta d\gamma d\bth $ - a constant that depends only on $k$ and $d$.
Next, we  make one final change of variables -  $\Lambda\tau^d \to t$, which implies that $dt = d\cdot\Lambda \tau^{d-1}d\tau$ and thus,
\begin{equation}\label{eq:mean_C_k_r}
	\mean{C_k(r)} = \frac{\tilde D_k}{d} n \int_{0}^\Lambda t^{k-1}e^{-t}dt.
\end{equation}
The latter integral is known as the \emph{lower incomplete gamma function} and has a closed form expression which yields,
\[
	\mean{C_k(r)} = \frac{\tilde D_k}{d} n (k-1)!\param{1-e^{-\Lambda}\sum_{j=0}^{k-1}\frac{\Lambda^j}{j!}}.
\]
Finally, denoting 
\begin{equation}\label{eq:D_k}
D_k = \frac{\tilde D_k (k-1)!}{d }
\end{equation} 
completes the proof.
\end{proof}


Remark: The conditions in Lemma \ref{lem:gen_crit_point} can  also be phrased by saying that the points in $\cY$ generate a critical point if their Voronoi cells have a common intersection with their dual Delaunay cell. The Voronoi and Delaunay tessellation have been extensively studied in the stochastic geometry literature. In particular, for Poisson-Delaunay complexes, it is known that circumspheres of the $d$ dimensional cells (which correspond to the radius of index $d$ critical points) have a gamma distribution \cite{moller_lectures_2012}. This agrees with the form of the integral in \eqref{eq:mean_nk_2}, from which the incomplete gamma function in the Proposition comes from.



\begin{proof}[Proof of Corollary \ref{cor:mean_ec}]
First note that since $\T^d$ is a $d$-dimensional manifold and $\cU(n,r)\subset \T^d$,
then for $k>d$ we have $H_k(\cU(n,r)) = 0$. For $r<r_{\max}$, by the Nerve lemma \ref{lem:nerve} this is true for $\cC(n,r)$ as well. Therefore,
\begin{equation}\label{eq:ec_cech_betti}
\chi(r) = \sum_{k=0}^d (-1)^k \beta_k(r).
\end{equation}
One consequence of Morse theory (cf. \cite{milnor_morse_1963}) is that the sum of Betti numbers in \eqref{eq:ec_cech_betti} can be replaced by
\[\begin{split}
	\chi(r) = \sum_{k=0}^d (-1)^k C_k(r).
	\end{split}
\]
Using Proposition \ref{prop:mean_C_k_r}, and the fact that $\mean{C_0(r)} = n$ for every $r>0$ (all the minima are at distance $0$), we have for every $r \le r_{\max}$
\[
\begin{split}
	\mean{\chi(r)} &= n + \sum_{k=1}^d (-1)^{k} \param{D_k n \param{1-e^{-\Lambda}\sum_{j=0}^{k-1}\frac{\Lambda^j}{j!}}} \\
	&= n\param{1+\sum_{k=1}^d(-1)^kD_{k}}+ne^{-\Lambda}\sum_{j=0}^{d-1}\sum_{k=j+1}^d(-1)^{k-1} D_k \frac{\Lambda^j}{j!}.
\end{split}
\]
Denoting 
\begin{equation}\label{eq:A_j}
A_j = \frac{1}{j!}\sum_{k=j+1}^{d}(-1)^{k-1} D_k
\end{equation}
we have
\begin{equation}\label{eq:mean_ec_full}
\mean{\chi(r)} = n\param{1-A_0}+ne^{-\Lambda}\sum_{j=0}^{d-1}A_j {\Lambda^j}
\end{equation}
Taking $\Lambda\to \infty$, from Corollary \ref{cor:betti_bounds} we have that in this case $\meanx{\beta_k(r)} \to \beta_k(\T^d)= \binom{d}{k}$. Using \eqref{eq:ec_betti} we then have
\[
	\limninf \mean{\chi(r)} =\chi(\T^d) =  \sum_{k=0}^d (-1)^k \binom{d}{k} = 0.
\]
On the other hand, as $\Lambda\to\infty$ (and $n\to\infty$) the second term in \eqref{eq:mean_ec_full} vanishes while the first keeps growing. In order for the limit to be zero, we therefore must have  $A_0 = 1$.
That completes the proof. 

\end{proof}


\subsection{Proofs for Section \ref{sec:theta_cycles}} \label{sec:proofs_7}


Lemma \ref{lem:theta_cycle} is central to the proof for the existence of cycles.
However, its proof is more algebraic topological and in a different spirit than the rest of this paper.  The details of this argument are not used elsewhere in the paper, and an intuitive explanation was presented in Section \ref{sec:theta_cycles}.


In the following proof we will be using both simplicial homology (for the \cech complex $\cC(\cP,r)$) and singular homology (for the union of balls $\cU(\cP,r)$). An introduction to both can be found in \cite{hatcher_algebraic_2002,munkres_elements_1984}. In both cases, we denote
\begin{itemize}
\item $C_k$ - the chain group, i.e.~all formal sums of simplexes (should not be confused with $C_k(r)$ used in the paper),
\item $\partial_k:C_k\to C_{k-1}$ - the boundary operator,
\item $Z_k$ - the cycles groups consisting of all chains $\gamma\in C_k$ such that $\partial_k \gamma = 0$,
\item $B_k$ - the group of boundary chains,  i.e.~all the chains $\gamma\in C_k$ such that $\gamma = \partial_{k+1} \hat\gamma$ for some $\hat\gamma\in C_{k+1}$,
\item $H_k$ - the $k$-th homology group defined by $H_k = Z_k/B_k$.
\end{itemize}

Finally, recall the definitions in \eqref{eq:A_eps} and \eqref{eq:epsilon} - 
\[
	A_\eps(\cY) := B_{R(\cY)}(C(\cY)) \backslash B^\circ_{\eps R(\cY)}(C(\cY)),
\]
and
\[
\phi = \phi(\cY) := \frac{\inf_{x\in\partial \Delta(\cY)}\norm{x-C(\cY)}}{2R(\cY)}.
\]

\begin{proof}[Proof of Lemma \ref{lem:theta_cycle}]\ \\
We want to show that for a critical point of index $k$ generated by $\cY$, if $A_\phi(\cY) \subset \cU(\cP,  R(\cY))$, then a $k$-cycle is created.
Consider the \cech complex $\cC_r := \cC(\cP,r)$. Since $R(\cY)$ is the critical value, then at $r=R(\cY)$ a single new $k$-simplex is added to the complex, call it $\Delta$. 
Let $r^-<r$ be a radius close enough to $r$ such that $\cC_{r^-} = \cC_r\backslash\Delta$  and in particular such that $\partial \Delta\in \cC_{r^-}$. Recall also that for a critical point to exist we require that $B(\cY) \cap \cP = \emptyset$, and therefore $\Delta$ is not a face of any $(k+1)$-simplex in $\cC_r$ (by the construction of the \cech complex).

The boundary $\partial \Delta$ is a cycle in $Z_{k-1}(\cC_{r^-})$. Suppose that  $\partial \Delta\in B_{k-1}(\cC_{r^-})$, then $\partial \Delta = \partial \gamma$ for some $\gamma\in C_k(\cC_{r^-})$. Taking $\gamma' = \Delta-\gamma$ then $\partial\gamma' = 0$, so $\gamma'\in Z_k(\cC_r)$. However, since $\Delta$ is not a face of any $(k+1)$-simplex in $\cC_r$, we must have $\gamma'\not\in B_k(\cC_r)$ and therefore $\gamma'$ is a nontrivial $k$-cycle in $H_k(\cC_r)$. This cycle is not homologous to any cycle in $\cC_{r^-}$ since it includes $\Delta$. Thus, to complete the proof all we have to show is that $\partial \Delta \in B_{k-1}(\cC_{r^-})$.

Recall that every $k$-simplex in the \cech complex $\cC_{r}$ is of the form $\sigma = [x_0,\ldots, x_k]$, where $x_i\in \cP$ and
\[
	B_r(x_0)\cap\cdots\cap B_r(x_k) \ne \emptyset.
\]
Since we only consider $k\le d$, we can embed every abstract $k$-simplex $\sigma\in \cC_r$ back into a $k$-simplex in $\R^d$, and denote its embedding by $\pi(\sigma)$. Setting $\cU_r := \cU(\cP,r)\subset\R^d$, there is a natural map from $C_k(\cC_r)$ - the group of simplicial chains - to $C_k(\cU_r)$ - the group of singular chains - that comes out from the Nerve Lemma \ref{lem:nerve} (using the embedding $\pi$).
Let $\imath_k:C_k(\cC_{r^-})\to C_k(\cU_{r^-})$ be the natural map between the simplicial and singular chains. The induced map $h_k:H_k(\cC_{r^-})\to H_k(\cU_{r^-})$ is an isomorphism by the Nerve Lemma \ref{lem:nerve}. If we can show that $h_{k-1}(\partial\Delta) = 0\in H_{k-1}(\cU_{r^-})$ then that will imply that $\partial\Delta = 0\in H_{k-1}(\cC_{r^-})$, and we are done. 

Finally, note that since we assume $A_\phi = A_\phi(\cY) \subset \cU_r$, and since $A_{2\phi}\subset A_{\phi}$, we can choose $r'$ to be close enough to $r$ so that $A_{2\phi} \subset \cU_{r'}$.
Our choice of $ \phi(\cY)$ implies that $\pi(\partial \Delta)\subset A_{2\phi}$, and therefore $\imath_{k-1}(\partial\Delta) \in C_{k-1}(A_{2\phi})$. Since $A_{2\phi}$ is a $d$-dimensional annulus, we have $H_{k-1}(A_{2\phi}) = 0$, and therefore $\imath_{k-1}(\partial\Delta) \in B_{k-1}(A_{2\phi})\subset B_{k-1}(\cU_{r^-})$.
Thus $h_{k-1}(\partial \Delta) = 0$, which completes the proof.
\end{proof}


Next, recall the definition of $\beta_k^\eps(r)$ in Section \ref{sec:theta_cycles}, and define -
\[
\begin{split}
	h_r^\eps(\cY) &:= (h_{r}(\cY)-h_{r'}(\cY))\ind\set{\phi(\cY) \ge \eps},\\
	g_r^\eps(\cY,\cP) &:= h_r^\eps(\cY)\ind\set{B_{r''}(C(\cY))\cap \cP_n = \cY}\ind\set{A_\eps(\cY) \subset \cU(\cP, R(\cY))},
	\end{split}
\]
where
\[
	\quad r' = r(1-\delta), \quad r'' = r(1+\sqrt{2\delta}), \quad \delta = \Lambda^{-2},
\]
and $h_r$ is defined in \eqref{eq:h_funcs}. Then we have
\[
	\beta_k^\eps(r) = \sum_{\cY\subset \cP_n} g_r^\eps(\cY,\cP_n).
\]


\begin{proof}[Proof of Lemma \ref{lem:theta_cycle_bound}]\ \\
We need to show that $\beta_k(r) \ge \beta_k^\eps(r)$. If $g^\eps_r(\cY,\cP_n) =1 $ then according to Lemma \ref{lem:theta_cycle} we have that $\cY$ creates a $\Theta$-cycle at radius $R(\cY) \in (r',r]$. To complete the proof we will show that this $\Theta$-cycle created prior to $r$ still exists at $r$ and therefore contributes to $\beta_k(r)$.

In the proof of Lemma \ref{lem:theta_cycle} we saw that when a $\Theta$-cycle is created  a  new $k$-simplex $\Delta$ is added to the complex. This simplex is ``isolated" in the sense that it is not a face of any $(k+1)$-simplex. As long as $\Delta$ remains isolated, the cycle it creates cannot become a boundary (a homologically trivial cycle). Thus, it remains to show that $g^\eps_r(\cY,\cP_n) =1$ implies that $\Delta$ is isolated at radius $r$.

For the $k$-simplex $\Delta$ generated by a set $\cY$ to be isolated at radius $r$, we need to verify that $\cap B_r(\cY) := \bigcap_{Y\in\cY}B_r(Y) $, will not intersect with any of the other balls of radius $r$. We will show in Lemma \ref{lem:intersect} that the distance between points in $\cap B_r(\cY)$ and $C(\cY)$ is bounded by $\sqrt{r^2-R^2(\cY)} \le \sqrt{r^2-r'^2} \le r\sqrt{2\delta}$. Thus, by the triangle inequality, if the ball of radius $r''=r(1+\sqrt{2\delta})$ around $C(\cY)$ contains only the points in $\cY$, there is no other ball of radius $r$ that intersects $\cap B_r(\cY)$. Since $g^\eps_r(\cY,\cP_n)=1$ requires that $B_{r''}(C(\cY))\cap \cP_n = \cY$ we completed the proof.

\end{proof}
%
%
%

\begin{lem}\label{lem:intersect}
For every $x\in\cap \B_r(\cY)$ we have $\norm{x-C(\cY)} \le \sqrt{r^2-R^2(\cY)}$.
\end{lem}
\begin{proof}
Let $\cY = \set{x_0,\ldots,x_k}$,  $c = C(\cY)$, and $R=R(\cY)$, so that $\norm{c-x_i} = R$ for all $i$.  Since $c \in \Delta(\cY)$ we can write $c = \sum_{i=0}^k \alpha_i x_i$, where $\alpha_i\ge 0$ and  $\sum_{i=0}^k \alpha_i = 1$. Let $x \in \cap B_r(\cY)$ then $\norm{x-x_i} \le r$ for every $i$. Now,
\[
	\norm{x-x_i}^2 = \norm{x-c + c-x_i}^2 = \norm{x-c}^2 + \norm{c-x_i}^2 - 2\iprod{x-c,c-x_i},
\]
where $\iprod{\cdot,\cdot}$ is the inner product in $\R^d$. Thus, we conclude that for every $i$
\[
	\norm{x-c}^2 \le r^2 - R^2 +2\iprod{x-c,c-x_i}.
\]
Summing over all the points $x_i$, we have
\[
	\sum_{i=0}^k\alpha_i\norm{x-c}^2  \le \sum_{i=0}^k\alpha_i(r^2 - R^2) +2	\sum_{i=0}^k\alpha_i\iprod{x-c,c-x_i}.
\]
Since $\sum_i \alpha_i = 1$ and $\sum_i\alpha_ix_i = c$ we have
\[
	\norm{x-c}^2 \le r^2-R^2.
\]
\end{proof}

Next, we prove Lemma \ref{lem:theta_cycle_order}.


\begin{proof}[Proof of Lemma \ref{lem:theta_cycle_order}]\ \\
We need to show that $\mean{\beta_k^\eps(r)} \sim n\Lambda^{k-2}e^{-\Lambda}$.
We start by fixing $\eps>0$ and evaluating the expectation of $\beta_k^{\eps}(r)$, similarly to the computations in the proofs of Proposition \ref{prop:mean_C_k_r}.
\begin{equation}\label{eq:mean_C_k_th}
\begin{split}
	\mean{\beta_k^\eps(r)} &= \frac{n^{k+1}}{(k+1)!}\mean{g_r^\eps(\cY', \cY'\cup\cP_n)}\\
	&= \frac{n^{k+1}}{(k+1)!} \int_{(\T^d)^{k+1}} h_r^\eps(\bx) e^{-n \omega_d(r'')^d}p_{\eps}(\bx)d\bx,
\end{split}
\end{equation}
where
\[
p_{\eps}(\bx) := \cprob{A_\eps(\cY') \subset \cU(\cY'\cup\cP_n, R(\cY'))}{\cY'=\bx,\  \cP_n\cap B_{r''}(C(\bx)) = \emptyset}.
\]
We will first show that $p_\eps(\bx)\to 1$ uniformly for all $\bx$.
Denoting $\P_{_\emptyset}(\cdot) := \cprob{\cdot}{\cP_n\cap B_{r''}(C(\bx)) = \emptyset}$,  we observe that 
\[ 
p_{\eps}(\bx) \ge \P_{_{\emptyset}}(A_\eps(\bx) \subset \cU(\cP_n, R(\bx))).
\]
In the following we use the shorthand notation:
$$R=R(\bx), A_\eps = A_\eps(\bx) , B = B_{r''}(C(\bx)), p_{\eps} = p_{\eps}(\bx).$$
Let $\cS$ be a $(\eps R/2)$-net of $A_\eps$, i.e.~for every $x\in A_\eps$ there exists $s\in \cS$ such that $\norm{x-s}\le\eps R/2$. Note that there exists $c>0$ such that 
\[	
\abs{\cS} \le c\frac{\vol(A_\eps)}{(\eps R/2)^d} = c \frac{\omega_dR^d(1-\eps^d)}{(\eps R/2)^d} = c\omega_d (2/\eps)^d(1-\eps^d) =: c_1.
\]
By the triangle inequality, if for every $s\in \cS$ we have $\cP_n\cap B_{R(1-\eps/2)}(s)\ne \emptyset$,
then  $A_\eps\subset \cU(\cP_n, R)$.
Therefore, 
\[
	p_{\eps} \ge \P_{_\emptyset}({\forall s\in \cS : \cP_n\cap B_{R(1-\eps/2)}(s)\ne \emptyset}).
\]
Note that under $\P_{_\emptyset}$ we have already $B\cap \cP_n = \emptyset$. Thus, for any $s\in\cS$
\[
	\P_{_\emptyset}(\cP_n\cap B_{R(1-\eps/2)}(s) = \emptyset) = e^{-n \vol(B_{R(1-\eps/2)}(s)\backslash B)}.
\]
Recall that $r''$ is the radius of $B$, and notice that since $R>r'$, we have that $r'' / R < r''/r' \to 1$. In particular, for $n$ large enough we have $r'' < R(1+\eps/4)$ (for all $\bx$).
Since $s\in A_\eps$, then at least a fixed proportion of the volume of $B_{R(1-\eps/2)}(s)$ must  be outside $B$ such that
\[
	\vol(B_{R(1-\eps/2)}(s)\backslash B) \ge c_2 \omega_d R^d,
\]
for some $c_2\in(0,1)$. Therefore,
\[
\P_{_\emptyset}(\exists s\in \cS : \cP_n\cap B_{R(1-\eps/2)}(s)= \emptyset) \le \sum_{s\in\cS} \P_{_\emptyset}({\cP_n \cap B_{R(1-\eps/2)}(s) = \emptyset})\le c_1e^{- c_2n\omega_d R^d}.
\]
For every $\bx$ such that $h_r^\eps(\bx)\ne 0$ we have $R(\bx)> r'$, and therefore   $1-c_1e^{- c_2\omega_d n(r')^d}\le p_\eps(\bx)\le 1$, Since $\omega_dn(r')^d \to \infty$, we have $p_\eps(\bx)\to 1$ uniformly for every $\bx$. 
Going back to \eqref{eq:mean_C_k_th} we therefore have
\[
\mean{\beta_k^\eps(r)} \approx \frac{n^{k+1}}{(k+1)!} e^{-\omega_d n(r'')^d}\int_{(\T^d)^{k+1}} h_r^\eps(\bx) d\bx.
\]

Using the same change of variables we had in the proof of Proposition \ref{prop:mean_C_k_r}, similarly to \eqref{eq:mean_nk_2} we have
\[
	\meanx{\beta_k^\eps(r)} \approx  D^\eps_k ne^{-\omega_d n(r'')^d} \Lambda^k \int_{1-\delta}^{1}  \tau^{dk-1} d\tau,
\]
where 
\[
 D_k^\eps := \frac{1}{\omega_d^k(k+1)!}\int_{\theta,\gamma,\bth} h^\eps(0,\theta+\gamma\circ \bth) J(1,\theta,\gamma,\bth)  d\theta d\gamma d\bth,
\]
and $h^\eps := h(\cY)\ind\set{\phi(\cY) \ge \eps}$. Recall that the points in the last integral are arranged on a sphere of radius $1$ centered at $\theta$. Thus,  \eqref{eq:epsilon} turns into
\[
	\phi(0,\theta+\gamma\circ \bth) = \inf_{x\in \partial\Delta(0,\theta+\gamma\circ \bth)}\norm{x-\theta}.
\]
This function is continuous in $\theta,\gamma,\bth$, its minimum value is zero and it is achieved when the center of the sphere ($\theta$) lies on one of the faces of the simplex $\Delta(0,\theta+\gamma\circ \bth)$. This  set  has measure zero, and therefore there exists $\eps_k >0$ such that the set $\set{\theta,\gamma,\bth : \phi(0,\theta+\gamma\circ\bth)  > \eps_k}$ has a positive measure. Taking $\eps = \eps_k$ we have that $ D_k^\eps > 0$ and since $\int_{1-\delta}^{1}\tau^{dk-1}d\tau =  \delta+o(\delta)$ we have
\begin{equation}\label{eq:mean_theta_bound}
	\mean{\beta_k^\eps(r)} \approx D_k^\eps \delta  n \Lambda^{k} e^{- \omega_d n(r'')^d} .
\end{equation}
Recall that $\delta = \Lambda^{-2}$, then
\[
\begin{split}\omega_d n(r'')^d &= \Lambda(r''/r)^d = \Lambda\param{1+\sqrt{2\delta}}^d = \Lambda\param{1+d\sqrt{2\delta} + o\param{\sqrt{2\delta}}} = \Lambda + O(1).
\end{split}
\]
Therefore,
\[
	\mean{\beta_k^\eps(r)} \approx D_k^\eps   n \Lambda^{k-2} e^{- \Lambda - O(1)}.
\]
So 
\[
	\mean{\beta_k^\eps(r)} \sim n \Lambda^{k-2} e^{- \Lambda }.
\]
That completes the proof.
\end{proof}


We will end this section with the proof for Lemma \ref{lem:theta_cycle_vanish}.


\begin{proof}[Proof of Lemma \ref{lem:theta_cycle_vanish}]\ \\
From Lemma \ref{lem:theta_cycle_order} we have that if $\Lambda = \log n + (k-2)\log\log n -w(n)$ then $\mean{\beta_k^\eps(r)} \to \infty$. If we show in addition that $\var{\beta_k^\eps(r)} \ll {\mean{\beta_k^\eps(r)}}^2$, then using Chebyshev's inequality we have $\prob{\beta_k^\eps(r) > 0} \to 1$ (cf. Chapter 4 in \cite{alon_probabilistic_2004}).
We therefore need to find an upper bound on the variance. Recall that
\[
	\varx{\beta_k^\eps(r)} = \meanx{(\beta_k^\eps(r))^2} - \meanx{ \beta_k^\eps(r)}^2.
\]
Since $\beta_k^\eps = \sum_{\cY\subset\cP_n}g_r^{\eps}(\cY,\cP_n)$, we can write $(\beta_k^\eps(r))^2$ as
\[\begin{split}
	(\beta_k^\eps(r))^2 &= \sum_{\cY_1\subset\cP_n}\sum_{\cY_2\subset\cP_n} g^\eps_r(\cY_1,\cP_n) g^\eps_r(\cY_2, \cP_n) \\
	&= \sum_{j=0}^{k+1}\sum_{\abs{\cY_1\cap\cY_2} = j}g^\eps_r(\cY_1,\cP_n) g^\eps_r(\cY_2, \cP_n),
	\end{split}
\]	
where $\cY_1,\cY_2$ iterate over all subsets of $\cP_n$ with $(k+1)$ points.

Denoting the inner sum by $I_j$, observe that $I_{k+1} = \beta_k^\eps(r)$, and therefore we have
\[
	\varx{\beta_k^\eps(r)} = \meanx{\beta_k^\eps(r)} + \sum_{j=1}^k\mean{I_j} + (\mean{I_0} - \meanx{\beta_k^\eps(r)}^2).
\]
To complete the proof we therefore need to show that  $\mean{I_j}/ \meanx{\beta^\eps_k(r)}^2 \to 0$ for every $1\le j \le k$, as well as $(\mean{I_0} - \meanx{\beta_k^\eps(r)}^2)/\meanx{\beta^\eps_k(r)}^2 \to 0$.

For $1\le j \le k$, using Palm theory (see Corollary \ref{cor:prelim:palm2}) we have that
\[
	\mean{I_j} = \frac{n^{2k+2-j}}{j!((k+1-j)!)^2} \mean{g^\eps(\cY_1',\cY'\cup\cP_n) g^\eps_r(\cY_2', \cY'\cup \cP_n)},
\]
where $\cY_1',\cY_2'$ are sets of $k+1$ random variables,  $\abs{\cY_1'\cap\cY_2'} = j$,  such that $\cY' = \cY_1'\cup \cY_2'$ is a set of $(2k+2-j)$ $\iid$ random variables, uniformly distributed in $\T^d$, and independent of $\cP_n$. Since $r< r_{\max}$ we have,
\begin{equation}\label{eq:I_j_ineq}
	\mean{I_j} \le \frac{n^{2k+2-j}}{j!((k+1-j)!)^2}\int_{(\T^d)^{2k+2-j}} h_r^\eps(\bx_1)  h^\eps_r(\bx_2) e^{-nV(\bx_1, \bx_2)}d\bx,
\end{equation}
where
\[
\begin{split}
\bx &= (x_0,\ldots, x_{2k+1-j}) \in (\T^d)^{2k+2-j},\\
\bx_1 &= (x_0,\ldots, x_{k}) \in (\T^d)^{k+1},\\
\bx_2 &= (x_0,\ldots, x_{j-1}, x_{k+1},\ldots, x_{2k+1-j})\in (\T^d)^{k+1},\\
V(\bx_1, \bx_2) &= \vol(B(\bx_1)\cup B(\bx_2)).\\
\end{split}
\]
Note that \eqref{eq:I_j_ineq} is an inequality, since we dropped the term verifying the coverage of the $A_\eps$-s, as well as the term verifying that $\bx_1 \cap B(\bx_2)=\emptyset$ and vice versa (both terms are bounded by $1$).
Notice that 
\[
	V(\bx_1,\bx_2) = \omega_d R^d(\bx_1) + \omega_d R^d(\bx_2) - \vol(B(\bx_1)\cap B(\bx_2)).
\]
Let $V_r(\Delta)$ denote the volume of the intersection of two balls of radius $r$ such that their centers are at distance $\Delta$ from each other. Since $R(\bx_i) \le r$ we have
\[
	\vol(B(\bx_1)\cap B(\bx_2)) \le V_r(\norm{C(\bx_1)-C(\bx_2)}).
\]
In addition, we know that $R(\bx_i) \ge (1-\delta)r$, and therefore,
\[\begin{split}
\mean{I_j}&\le \frac{n^{2k+2-j}}{j!((k+1-j)!)^2}e^{-\Lambda(2(1-\delta)^d-1)}\int_{(\T^d)^{2k+2-j}} h_r^\eps(\bx_1)  h^\eps_r(\bx_2) e^{-\Lambda+ V_r(\norm{C(\bx_1)-C(\bx_2)})}d\bx\\
&\approx \frac{n^{2k+2-j}}{j!((k+1-j)!)^2}e^{-\Lambda}\int_{(\T^d)^{2k+2-j}} h_r^\eps(\bx_1)  h^\eps_r(\bx_2) e^{-\Lambda+ V_r(\norm{C(\bx_1)-C(\bx_2)})}d\bx,
\end{split}
\]
since $\delta = \Lambda^{-2}$.
Since the sets $\bx_1,\bx_2$ have at least one point in common ($x_0$), we can apply a similar change of variables to the one used in the proof of Proposition \ref{prop:mean_C_k_r} and have
\[
\mean{I_j} \le \frac{n^{2k+2-j}r^{d({2k+1-j})}}{j!((k+1-j)!)^2} e^{-\Lambda}\int_{(\R^d)^{2k+1-j}} h^\eps_1(0,\by_1) h^\eps_1(0,\by_2) e^{-\Lambda + nr^d V_1(\norm{C(0,\by_1)- C(0,\by_2)})}d\by,
\]
where
\[
\begin{split}
\by &= (y_1,\ldots, y_{2k+1-j}) \in (\R^d)^{2k+1-j},\\
\by_1 &= (y_1,\ldots, y_{k}) \in (\R^d)^{k},\\
\by_2 &= (y_1,\ldots, y_{j-1}, y_{k+1},\ldots, y_{2k+1-j})\in (\R^d)^{k}.
\end{split}
\]
Let $S := \set{\by: h_1^\eps(0,\by_1) = h_1^\eps(0,\by_2) = 1}$, then $S$ is bounded and we can write
\begin{equation}\label{eq:bound_S}
\mean{I_j} \le c_1 n\Lambda^{2k+1-j} e^{-\Lambda}\int_{S} e^{-\Lambda + nr^d V_1(\norm{C(0,\by_1)- C(0,\by_2)})}d\by,
\end{equation}
for some constant $c_1>0$.
Next, fix $\alpha >0$, and define 
\[
\begin{split}
S_1 &= \set{\by\in S: \norm{C(0,\by_1)- C(0,\by_2)} \le \alpha},\\
S_2 &= S\backslash S_1 = \set{\by\in S: \norm{C(0,\by_1)- C(0,\by_2)} >\alpha}.
\end{split}
\]
We will bound the integral on both sets separately. 
On the set $S_1$, note first that $e^{-\Lambda + nr^d V_1(\norm{C(0,\by_1)- C(0,\by_2)})} \le 1$, and therefore
\[\begin{split}
&\int_{S_1} e^{-\Lambda + nr^d V_1(\norm{C(0,\by_1)- C(0,\by_2)})}d\by \\
&\quad\le \int_{\by_1} h_1^\eps(0,\by_1)\int_{\by_2\backslash \by_1} h_1^\eps(0,\by_2)\ind\set{\norm{C(0,\by_1)-C(0,\by_2)}\le\alpha}d\by.
\end{split}
\]
For a fixed $\by_1$, if $\norm{C(0,\by_1)-C(0,\by_2)} \le \alpha$, and $R(0,\by_2) \in (1-\delta,1)$ then for every $y\in \by_2\backslash \by_1$ we have
\[
	1-\delta-\alpha \le \norm{y-C(0,\by_1)} \le 1+\alpha.
\] 
Thus, the $(k+1-j)$ points in $\by_2\backslash\by_1$ lie in a thin annulus, and we have
\[\begin{split}
\int_{\by_2\backslash \by_1} h_1^\eps(0,\by_2)\ind\set{\norm{C(0,\by_1)-C(0,\by_2)}\le\alpha}d(\by_2\backslash \by_1) &\le \param{\omega_d((1+\alpha)^d - (1-\delta-\alpha)^d)}^{k+1-j}\\
&\approx (\omega_d(2d\alpha+d\delta))^{k+1-j},
\end{split}
\]
where we used the approximation $(1+x)^d \approx 1+dx$ (which will be justified later when we take $\alpha\to 0$). Since $R(0,\by_1)\in (1-\delta,1)$ we have that $\int_{\by_1} h_1^\eps(0,\by_1) =O(\delta) = O(\Lambda^{-2})$ and therefore,
\begin{equation}\label{eq:bound_S_1}
\int_{S_1} e^{-\Lambda + nr^d V_1(\norm{C(0,\by_1)- C(0,\by_2)})}d\by \le c_2\Lambda^{-2} (2d\alpha+d\delta)^{k+1-j},
\end{equation}
for some constant $c_2>0$.

On $S_2$ we use the fact that (see Appendix \ref{sec:vol_intersection}):
\begin{equation}\label{eq:vol_intersection}
V_1(\Delta) = \omega_d - \omega_{d-1}\Delta + o(\Delta),
\end{equation}
and 
together with the fact $V_1$ is decreasing in $\Delta$, we have
\begin{equation}\label{eq:bound_S_2}
\int_{S_2} e^{-\Lambda + nr^d V_1(\norm{C(0,\by_1)- C(0,\by_2)})}d\by \le \vol(S_2) e^{-\omega_{d-1}nr^d\alpha +o(nr^d\alpha)} \le \vol(S_2) e^{-c_3\Lambda\alpha},
\end{equation}
for some $c_3>0$.
Combining \eqref{eq:bound_S}, \eqref{eq:bound_S_1} and \eqref{eq:bound_S_2}, we have that
\[
	\mean{I_j}\le c_4 n\Lambda^{2k+1-j} e^{-\Lambda} (\Lambda^{-2}(2\alpha+\delta)^{k+1-j} + e^{-c_3\Lambda\alpha}),
\]
for some constant $c_4 > 0$.
From Lemma  \ref{lem:theta_cycle_order} we have that
\[
	\meanx{\beta_k^\eps(r)} \sim n\Lambda^{k-2} e^{-\Lambda},
\]
and therefore
\[
	\frac{\mean{I_j}}{\meanx{\beta_k^\eps(r)}^2} \le c_4 n^{-1}\Lambda^{5-j}e^\Lambda (\Lambda^{-2}(2\alpha+\delta)^{k+1-j} + e^{-c_3\Lambda\alpha}).
\]
For $\Lambda = \log n + (k-2)\log\log n -w(n)$ we have
\[
	\frac{\mean{I_j}}{\meanx{\beta_k^\eps(r)}^2} \le c_5 e^{-w(n)}\param{(\log n)^{k+1-j} (2\alpha+\delta)^{k+1-j} + (\log n)^{k+3-j}e^{-c_3\log(n)\alpha}},
\]
for some $c_5>0$. Taking
\[
	\alpha = \frac{k+3-j}{c_3}\frac{\log \log n}{\log n},
\]
and assuming that $w(n) \gg \log\log\log n$ we have
\[
\limninf\frac{\mean{I_j}}{\meanx{\beta_k^\eps(r)}^2} = 0.
\]

To complete the proof, it remains to show that
\[
	\frac{\mean{I_0} - \meanx{\beta_k^\eps(r)}^2}{\meanx{\beta_k^\eps(r)}^2} \to 0.
\]
Using Palm theory, we have that
\[
	\mean{I_0} = \frac{n^{2k+2}}{((k+1)!)^2}\mean{g^\eps_r(\cY_1',\cY'\cup \cP_n)g^\eps_r(\cY_2',\cY'\cup \cP_n)},
\]
where $\cY_1',\cY_2'$ are two disjoint sets of $k+1$ $\iid$ random variables, uniformly distributed in $\T^d$, independent of $\cP_n$, and $\cY' = \cY_1'\cup \cY_2'$. On the other hand, we can rewrite $\meanx{\beta_k^\eps(r)}^2$ as
\[
	\meanx{\beta_k^\eps(r)}^2 = \frac{n^{2k+2}}{((k+1)!)^2} \mean{g^\eps_r(\cY_1',\cY_1'\cup\cP_n)g^\eps_r(\cY_2',\cY_2'\cup \cP_n')},
\]
where $\cY_1',\cY_2'$ are the same as above, and $\cP_n'$ is an independent copy of $\cP_n$. Next, define
\[
\Delta g := {g^\eps_r(\cY_1',\cY' \cup \cP_n)g^\eps_r(\cY_2',\cY' \cup \cP_n)-g^\eps_r(\cY_1',\cY_1'\cup \cP_n)g^\eps_r(\cY_2',\cY_2'\cup \cP_n')}
\]
then
\[
\mean{I_0} - \meanx{\beta^\eps_k(r)}^2 =  \frac{n^{2k+2}}{((k+1)!)^2} \mean{\Delta g}.
\]
We can split $\Delta g$ into two terms -
\[
\Delta_1 := \Delta g\cdot\ind\set{B(\cY_1')\cap B(\cY_2') = \emptyset}
\qquad
\Delta_2 := \Delta g\cdot \ind\set{B(\cY_1')\cap B(\cY_2') \ne \emptyset}.
\]
First, we show that $\mean{\Delta_1} = 0$. Observe that 
\begin{align*}
\Delta_1 &= {h}^\eps_r(\cY_1'){h}^\eps_r(\cY_2')\ind\set{B(\cY_1')\cap B(\cY_2')=\emptyset}\\
&\qquad \times \Big(\ind\set{\cP_n \cap B(\cY_1') = \emptyset}\ind\set{\cP_n \cap B(\cY_2') = \emptyset} \\
&\qquad\qquad - \ind\set{\cP_n \cap B(\cY_1') = \emptyset}\ind\set{\cP_n' \cap B(\cY_2') = \emptyset}\Big).
\end{align*}
If $\Delta_1 \ne 0$, then $B(\cY_1')$ and $B(\cY_2')$ must be disjoint. Therefore, given $\cY_1'$ and  $\cY_2'$, the set $\cP_n\cap B(\cY_2')$ is independent of the set $\cP_n\cap B(\cY_1')$ (by the spatial independence of the Poisson process), and has the same distribution as $\cP_n'\cap B(\cY_2')$. Thus, $\cmean{\Delta_1}{\cY_1',\cY_2'} = 0$, for any $\cY_1,\cY_2$. This implies that $\mean{\Delta_1} = 0$.
For $\Delta_2$, notice that
\[
	\Delta_2 \le g^\eps_r(\cY_1',\cY_1' \cup \cP_n)g^\eps_r(\cY_2',\cY_2' \cup \cP_n) \ind\set{B(\cY_1')\cap B(\cY_2') \ne \emptyset},
\]
and then
\[
	\mean{\Delta_2} \le \int_{(\T^d)^{2k+2}} {h}^\eps_r(\bx_1){h}^\eps_r(\bx_2) e^{-n V(\bx_1,\bx_2)}  \ind\set{B(\bx_1)\cap B(\bx_2) \ne \emptyset}d\bx_1 d\bx_2,
\]
where $\bx_1 = (x_0,\ldots,x_{k}), \bx_2 = (x_{k+1},\ldots, x_{2k+1})$. 
Note that here the sets $\bx_1,\bx_2$ do not share any point in common. However, since the balls intersect, the points have to be very close to each other, and  in particular to $x_0$. Thus, we can apply similar arguments to the ones we used for $I_j$ that to show that
\[
	\limninf\frac{\mean{I_0} - \meanx{\beta_k^\eps(r)}^2}{\meanx{\beta_k^\eps(r)}^2} = 0.
\]
That completes the proof.

\end{proof}


\section{Riemannian Manifolds}\label{sec:riemannian}


The arguments in this paper should extend to all compact Riemannian manifolds.
In this section we wish to sketch a heuristic argument for this.  We hope to give a detailed treatment (with more effective results) elsewhere.

For flat tori, we made use of the fact that for $r$ sufficiently small, the $r$-balls are all (1) embedded (i.e.~are topological balls) (2) flat  (isometric to Euclidean balls), and when a sufficiently large number of points are sampled uniformly - (3) all the critical points of the distance function are associated with points within distance $r$ of one another.  We have also used the fact that (4) intersections of balls at small scales are either empty or contractible.

For general compact Riemannian manifolds (1) and (3) are both true and straightforward.  (4) is also true when $r$ is smaller than the convexity radius. However, (2) is not true in general, because the curvature might be nonzero. 

However, the problems that we are considering here are asymptotic and scale invariant. 
This formally implies that flatness is irrelevant; curvature can always be assumed to be arbitrarily small.

More explicitly, scale $r$ on the Riemannian manifold $(\cM,g)$ is the same as scale $2r$ on the manifold $(\cM, 2g)$ whose curvature tensor is multiplied by $1/4$.  Moreover, the covariant derivatives of the curvature tensor are multiplied by even smaller negative powers of $2$.  As a result, by working with a small enough $r$, we can arrange (in our original manifold) for the set of $r$-balls to be a precompact set in any fixed $C^k$-topology (given by a bound on the $(k+1)$-st derivatives of the metric) in a small neighborhood of the flat $g_{i,j} = \delta_{i,j}$ metric.  A continuity argument then gives that for $n$ large enough (so that all the critical points are at distance smaller than $r$) all of the integrals involved in the computations of expected values and variances are as close as desired to the flat values.  The arguments of the paper then directly apply.


\section{Conclusion}\label{sec:conclusion}


In this paper we studied the phase transition describing the vanishing of homology in random \cech complexes constructed over the $d$-dimensional torus.
We saw that, as opposed to other types of random complexes, for the \cech complex there exist only two sharp phase transition - one for $H_0$ and one for the rest of the homology groups $H_k$, $k\ge 1$. 

A possible explanation for the difference between connectivity and higher-order homology is the following. To show that the \cech complex is connected, it is enough to show that there are no isolated points (See \cite{penrose_random_2003}). A point in the \cech complex is isolated if and only if the ball of radius $2r$ around it is empty (i.e. containing no other points from $\cP_n$). On the other hand, as we have seen in this paper, the vanishing of the rest of the homology groups is related to critical points, and a critical point with value $r$ requires that the ball of radius $r$ around it is empty.
In other words, connectivity is related to empty balls of radius $2r$ whereas higher homology is related to empty balls of radius $r$. Heuristically, this should imply that the connectivity radius should be half the radius needed for the vanishing of the higher homology. In fact the difference between $\Lambda_0 = 2^{-d}\log n$ and $\Lambda_k \approx \log n$ describes exactly that (recall that $\Lambda = \omega_d nr^d$).

In addition to the sharp phase transition at $\Lambda = \log n$ we proved that there is a  lower order scale such that the different homology groups vanish in an ascending order. The results in this paper show that for $H_k$ the vanishing threshold is in the range $[\log n + (k-2)\log\log , \log n + k\log\log n]$, and it remains a future work to determine the exact vanishing point within this range.


\section*{Acknowledgements}

We are very grateful to Robert Adler, Yuliy Baryshnikov, Fr{\'e}d{\'e}ric Chazal, Yogeshwaran Dhandapani, Matthew Kahle, Sayan Mukherjee, and Steve Smale for many helpful conversations about this and related work. We would also like to thank the anonymous referees for their useful comments.


\newpage
\section*{Appendix}


\appendix


\section{Palm Theory for Poisson Processes}


The following theorem will be very useful when computing expectations related to Poisson processes. 


\begin{thm}[Palm theory for Poisson processes]
\label{thm:prelim:palm}
Let $(X,\rho)$ be a metric space, $f:X\to\R$ be a probability density on $X$, and let $\cP_n$ be a Poisson process on $X$ with intensity $\lambda_n = n f$.
Let $h(\cY,\cX)$ be a measurable function defined for all finite subsets $\cY \subset \cX \subset X^d$  with $\abs{\cY} = k$. Then
\[
    \E\Big\{\sum_{ \cY \subset \cP_n}
    h(\cY,\cP_n)\Big\} = \frac{n^k}{k!} \mean{h(\cY',\cY' \cup \cP_n)}
\]
where $\cY'$ is a set of $k$ $\iid$ points in $X$ with density $f$, independent of $\cP_n$.
\end{thm}


For a proof of Theorem \ref{thm:prelim:palm}, see for example \cite{penrose_random_2003}.
We shall also need the following corollary, which treats second moments:


\begin{cor}\label{cor:prelim:palm2}
With the notation above, assuming $\abs{\cY_1} = \abs{\cY_2} = k$,
\[
    \E\Big\{\sum_{ \substack {
                    \cY_1 ,\cY_2\subset \cP_n  \\
                    \abs{\cY_1 \cap \cY_2} = j }}
    h(\cY_1,\cP_n)h(\cY_2,\cP_n)\Big\} = {\frac{n^{2k-j}}{j!((k-j)!)^2}} \mean{h(\cY_1',\cY' \cup \cP_n)h(\cY_2',\cY' \cup \cP_n)}
\]
where $\cY' = \cY'_1 \cup \cY'_2$ is a set of $2k-j$ $\iid$ points in $X$ with density $f$, independent of $\cP_n$, and $\abs{\cY_1'\cap\cY_2'} = j$.
\end{cor}


For a proof of this corollary, see for example \cite{bobrowski_distance_2014}.


\section{Coverage threshold}\label{sec:appendix_coverage}


The work in \cite{flatto_random_1977} studied the number randomly placed balls of radius $r$ needed in order  to cover each point in a compact Riemannian manifold $m$ times. Translating the results in \cite{flatto_random_1977} for $m=1$ using the notation in this paper, we have the following.
Let $N_r$ be the number of balls of radius $r$ required to cover a manifold $\cM$ (assuming that the centers are uniformly distributed in $\cM$). Let $\alpha = \omega_d r^d = n^{-1}\Lambda$, and define
\[
	X_r = \alpha N_r - \log(\alpha^{-1})-d\log\log(\alpha^{-1}).
\]


\begin{thm}[Theorem 1.1 in \cite{flatto_random_1977}]\label{thm:flatto}
There exists $r_1 > 0$ and $C>0$ such that
\begin{align}
\label{eq:coverage_th_1}\prob{X_r > x} \le Ce^{-x/8},\quad \forall x\ge 0, r\le r_1,\\
\label{eq:coverage_th_2}\prob{X_r < x} \le Ce^x, \quad \forall x\le 0, r\le r_1.
\end{align}
\end{thm}


We wish to use this result to prove the following phase transition.


\begin{cor}\label{cor:coverage}
Let $\cM$ be a compact Riemannian manifold, let $\cP_n$ be a homogeneous Poisson process on $\cM$ with intensity $n$, and let $w(n)\to\infty$.
\begin{enumerate}
\item If $\Lambda = \log n + (d-1)\log\log n + w(n)$ then
\[
		\cM \subset \cU(\cP_n,r)\quad w.h.p.
\]
\item If $\Lambda = \log n + (d-1)\log\log n - w(n)$ then
\[
		\cM \not\subset \cU(\cP_n,r)\quad w.h.p.
\]
\end{enumerate}
\end{cor}


\begin{proof}
Recall our definition of the Poisson process in Section \ref{sec:poisson} as $\cP_n = \set{X_1,\ldots,X_N}$ where $N\sim\pois{n}$.
For any $\tilde n>1$ (not necessarily integer) define $\cX_{\tilde n} = \set{X_1,\ldots, X_{\floor{\tilde n}}}$, and let $\tilde \Lambda = \omega_d \tilde n r^d$.
\begin{enumerate}
\item Note that $\prob{X_r > x} = \prob{\cM\not\subset \cU(\cX_{\tilde{n}},r)}$, where $x =  \alpha \tilde n -  \log(\alpha^{-1}) - d\log\log(\alpha^{-1})$.
Now, taking $\tilde\Lambda = \log \tilde n + (d-1)\log\log \tilde n +w({\tilde n})$, since $\alpha = \omega_d r^d = \tilde n^{-1}\tilde \Lambda$ we have
\[
x = (d-1)\log\log\tilde n +\log\tilde\Lambda -d\log(\log\tilde n-\log\tilde\Lambda)+w(\tilde n).
\]
Note that
\[
\lim_{\tilde n\to \infty}\param{ \log\tilde\Lambda-\log\log \tilde n} = \lim_{\tilde n \to \infty} \param{\log\log\tilde n - \log(\log\tilde n - \log\tilde\Lambda)} = 0.
\]
Therefore, we have $x= w({\tilde n}) + o(1)$, and using \eqref{eq:coverage_th_1} we have
\begin{equation}\label{eq:no_cover_prob}
	\prob{\cM\not\subset\cU(\cX_{\tilde n},r)} \le C e^{-(w({\tilde n})+o(1))/8}.
\end{equation}

Moving back to the Poisson process, let $q\in (1/2,1)$ and set $\tilde n = n-n^{q}$.
Then
\begin{equation}\label{eq:pois_no_cover}\begin{split}
	\prob{\cM \not\subset \cU(\cP_n,r)} &= \cprob{\cM \not\subset \cU(\cP_n,r)}{\abs{\cP_n} \ge \tilde n}\prob{\abs{\cP_n} \ge \tilde n} \\
	& + \cprob{\cM \not\subset \cU(\cP_n,r)}{\abs{\cP_n} < \tilde n}\prob{\abs{\cP_n} < \tilde n} \\
	&\le \prob{\cM\not\subset\cU(\cX_{\tilde n},r)} + \prob{\abs{\cP_n} < \tilde n}.
\end{split}
\end{equation}
Using Chebyshev's inequality we can show that $\prob{\abs{\cP_n} < \tilde n}\to 0$.
Now, take $\Lambda = \log n + (d-1)\log\log n + w(n)$, and notice that $\tilde\Lambda := \frac{\tilde n}{n}\Lambda = \Lambda - n^{q-1}\Lambda$. Since
\[
	\limninf(\log n - \log\tilde n) = 0, \quad and \quad
	\limninf(\log\log n -\log\log\tilde n) = 0,
\]	
we have that 
\[
	\tilde \Lambda = \log\tilde n + (d-1)\log\log\tilde n + w(n) + o(1).
\]
Combining \eqref{eq:pois_no_cover} and \eqref{eq:no_cover_prob} we then have
\[
	\prob{\cM \not\subset \cU(\cP_n,r)} \le Ce^{-(w(n)+o(1))/8} + o(1).
\]
Taking $w(n)\to\infty$ completes the proof.
\item Notice that $\prob{X_r< x} = \prob{\cM \subset\cU(\cX_{\tilde n},r)}$ where  $x = \alpha \tilde n - \log(\alpha^{-1})-d\log\log(\alpha^{-1})$.
Similarly to the above, taking $\tilde \Lambda = \log \tilde n + (d-1)\log\log\tilde n - w({\tilde n})$ and using \eqref{eq:coverage_th_2} we have that
\[
\prob{\cM\subset\cU(\cX_{\tilde n},r)} \le Ce^{-w({\tilde n})+o(1)},
\]
Next, set $\tilde n = n+n^q$ for some $q\in(1/2,1)$, then
\[\begin{split}
	\prob{\cM \subset \cU(\cP_n,r)} &= \cprob{\cM \subset \cU(\cP_n,r)}{\abs{\cP_n}> \tilde n}\prob{\abs{\cP_n} > \tilde n} \\
	& + \cprob{\cM \subset \cU(\cP_n,r)}{\abs{\cP_n} \le\tilde n}\prob{\abs{\cP_n} \le\tilde n} \\
	&\le  \prob{\abs{\cP_n} > \tilde n}+ \prob{\cM\subset\cU(\cX_{\tilde n},r)}.
\end{split}
\]
Continuing the same way as above yields
\[
\prob{\cM \subset \cU(\cP_n,r)} \le Ce^{-w(n)+o(1)} + o(1),
\]
and taking $w(n)\to \infty$ completes the proof.
\end{enumerate}
\end{proof}


\section{Intersection of Balls}\label{sec:vol_intersection}


In this section we provide a brief explanation to the approximation of the volume $V_1(\Delta)$ that appeared in \eqref{eq:vol_intersection}. 

The intersection of two balls is a region bounded by the union of two spherical caps. In \cite{li_concise_2011}  an explicit formula is shown for such volumes, from which we obtain 
\[
	V_1(\Delta) = 2\omega_{d-1} \int_0^{\cos^{-1}(\Delta/2)} \sin^d(\theta)d\theta.
\]
Thus,
\[
	V_1'(\Delta) = 2\omega_{d-1} \frac{-1}{2\sqrt{1-(\Delta/2)^2}} \sin^d(\cos^{-1}(\Delta/2)).
\]
At $\Delta=0$ the intersection is a whole ball, and therefore $V_1(\Delta)=0$, and from the last formula we have $V_1'(0) = -\omega_{d-1}$. Using taylor expansion we therefore have
\[
	V_1(\Delta) = \omega_d - \omega_{d-1}\Delta + o(\Delta).
\]	


\bibliographystyle{plain}
\bibliography{../../../../Latex/zotero}

\end{document}